\documentclass{amsart}
\usepackage[letterpaper,margin=1in]{geometry}
\usepackage{janeshi}

\theoremstyle{plain}
\newtheorem{theorem}{Theorem}[section]
\newtheorem{lemma}[theorem]{Lemma}
\newtheorem{proposition}[theorem]{Proposition}

\theoremstyle{definition}
\newtheorem{definition}[theorem]{Definition}

\theoremstyle{remark}
\newtheorem{remark}[theorem]{Remark}

\algtext*{EndIf}
\algtext*{EndFor}
\algtext*{EndWhile}
\algtext*{EndFunction}
\algtext*{EndProcedure}
\algtext*{EndLoop}
\algtext*{EndRepeat}
\algtext*{EndRepeatTimes}
\algtext*{EndRepeatUntil}

\newcommand{\loopsize}{486}

\begin{document}
\title{Lifting $L$-polynomials of genus $3$ curves}
\author{Jia Shi}
\begin{abstract}
Let $C$ be a smooth plane quartic curve over $\Q$. Costa, Harvey, and Sutherland provide an algorithm with an implementation, improving Harvey's average polynomial-time algorithm, to compute the $\bmod \ p$ reduction of the numerator of the zeta function of $C$ at all $p\leq B$, where $p$ is an odd prime of good reduction, in $O(B\log^{3+o(1)} B)$ time, which is $O(\log^{4+o(1)}p)$ time on average per prime. Alternatively, their algorithm can do this for a single prime $p$ of good reduction in $O(p^{1/2+o(1)})$ time. While this algorithm can be used to compute the full zeta function, no implementation of this step currently exists. 

In this article, we provide an algorithm and an implementation for the group operation on the Jacobian of $C$ over $\F_p$, where $p$ is an odd prime of good reduction. We provide a Las Vegas algorithm that takes the $\bmod \ p$ result of Costa, Harvey, and Sutherland's algorithm and uses it to compute the full zeta function.  The expected running time of the algorithm is bounded by $O(p^{1/4+o(1)})$. Our lifting algorithm can also be applied to hyperelliptic curves of genus 3.
\end{abstract}
\maketitle

\section{Introduction}
Let $C$ be a genus $3$ curve over $\Q$.
The $L$-function of $C$ is defined by an Euler product 
\[L(C,s) = \prod_{p} L_p(p^{-s})^{-1} = \sum_{n\geq 1} a_n n^{-s},\]
where $L_p$ is an integer polynomial of degree at most $6$.
The $L$-function $L(C,s)$ appears in many open conjectures in arithmetic geometry, including the paramodular conjecture and generalizations of the
Sato--Tate conjecture, the Birch and Swinnerton-Dyer conjecture, and the Riemann hypothesis.
To study these conjectures, we wish to compute $L(C,s)$ explicitly, which typically involves computing the polynomials $L_p(T)$ for all 
primes $p$ up to a bound $B$, which under standard conjectures can be taken to be $O(\sqrt{N})$, where $N$
is the conductor of the Jacobian of $C$.

If $p$ is a prime of good reduction for $C$, then $L_p(T)$ is the numerator of the zeta function 
\[
Z_{C_p}(T):= \exp\left(\sum_{n\ge 1} \frac{\# C(\F_{p^n})}{n}T^n\right) = \frac{L_p(T)}{(1-T)(1-pT)}
\]
of the reduction $C_p$ of $C$ modulo $p$.  While the definition of $Z_{C_p}(T)$ involves $\#C(\F_{p^n})$ for all $n\ge 1$, it depends only on $\#C(\F_{p^n})$ for $n\le g$. For $g=3$, it can be computed by counting points on $C_p$ over $\F_p$, $\F_{p^2}$, and $\F_{p^3}$. However, in this paper we are interested in methods that are more efficient than point counting. 
When $p$ is a prime of bad reduction, computing $L_p(T)$ is more difficult. But if one is willing to assume that $L(C,s)$ satisfies the expected functional equation, 
as implied by the Hasse--Weil conjecture, it is possible to deduce $L_p(T)$ using the knowledge of 
$L$-polynomials from sufficiently many primes of good reduction, which reduces the problem of computing $L(C,s)$ to the problem of computing $L_p(T)$ for sufficiently many good primes~$p$.

Harvey's general-purpose average polynomial-time algorithm \cite{Har15} for arithmetic schemes  provides an asymptotically efficient ``in principle'' algorithm to compute $L_p(T)$ for all primes $p\leq B$ of good reduction for $C$ in $O(B\log^{3+o(1)} B)$ time, 
which is $O(\log^{4+o(1)}p)$ time on average per prime. This is considerably faster than any algorithm 
known for computing $L_p(T)$ for a given prime, particularly in the generic case when $C$ is a smooth plane quartic, where existing methods such as Tuitman's algorithm \cite{Tui14, Tui16} and Costa's algorithm \cite{Cos15, CHK19} have complexities that are quasilinear in $p$.  Harvey's general-purpose method also provides an $O(p^{1/2+o(1)})$ time algorithm to compute $L_p(T)$ for a particular prime $p$.

While Harvey's method is very efficient asymptotically, a general-purpose practical implementation remains an open challenge. 
 In the case of genus 3 curves, however, there are very efficient implementations of the first step of Harvey's algorithm, 
 which computes $L_p(T)$ modulo $p$, including Harvey and Sutherland's implementation for hyperelliptic curves \cite{HS14,HS16,Sut20},
  Harvey, Massierer, and Sutherland's implementation for geometrically hyperelliptic curves \cite{HMS16}, and most recently, Costa, Harvey, and Sutherland's implementation for smooth plane quartics \cite{CHS23}.

The goal of this paper is to provide an efficient alternative to the second step of Harvey's algorithm, lifting $L_p(T)\bmod p$ to $L_p\in \Z[T]$, using group operations in the Jacobian of a genus 3 curve. Assuming we can perform group operations and hash group elements on the 
Jacobian of $C_p$ over $\F_p$ and $\F_{p^2}$, we present a probabilistic algorithm (of Las Vegas type) that takes $C_p$ and $L_p(T) \bmod p$  
as inputs and outputs $L_p(T)$ in ${O}(p^{1/4+o(1)})$ expected time.
When combined with the implementation in \cite{CHS23}, this provides an algorithm to compute $L_p(T)$ for all primes of
good reduction $p\leq B$ in $O(B^{5/4+o(1)})$ expected time.  While this is asymptotically slower than Harvey's $O(B^{1+o(1)})$ algorithm, it is asymptotically faster than existing implementations for smooth plane quartics (all of which have complexity $O(B^{2+o(1)})$), and also faster in practice, as can be seen in Table~\ref{Table-exp1}. We also provide an algorithm that computes $L_p(T)$ for a particular prime $p$ in $O(p^{1/2 + o(1)})$ expected time, matching the best known asymptotic result (Harvey's algorithm) and running faster in practice than any existing implementation (see Table~\ref{Table-exp2}).

Over a finite field, a genus 3 curve is either hyperelliptic or a smooth plane quartic.  For hyperelliptic curves, a general-purpose implementation of Cantor's algorithm can be found in Magma and other computer algebra systems. See \cite{Sut19b} for a particularly efficient implementation in the genus 3 case. 
For smooth plane quartics, Flon, Oyono, and Ritzenthaler \cite{FOR04} provide an efficient implementation of the group operation for curves with a rational flex when computing the sum of two typical divisors. However, not all smooth plane quartics over a finite field have a rational flex, and their algorithm does not cover the cases 
when one of the addends cannot be represented using a standard Mumford representation. Weinfeld \cite{Wei21} gives an algorithm for the group operation on $\Jac(C_p)(\F_{p^k})$ for any $k\geq 1$, as 
well as a baby-step giant-step method for lifting $L_p(T)\bmod p$
to integer coefficients. However, it is possible that this method does not narrow the candidates to a single one, which is 
the case that Steps~3 and~4 (Sections~\ref{section-step3} and~\ref{section-lv}) resolve.
The work in \cite{Wei21} does not provide an implementation.

In this paper, we provide an efficient algorithm and implementation for performing the group operation 
on the Jacobian of a smooth plane quartic curve $C_p$ over $\F_{p^k}$ for any $k\geq 1$. We also 
provide a method to uniquely identify group elements, which is crucial to the efficiency of our algorithm for lifting $L_p(T)$ (which relies on the baby-step giant-step method). To our knowledge, this is the first fully general implementation of the group operation on the Jacobian of
a smooth plane quartic with this property. 
Our algorithm is standalone, but can be sped up significantly when combined with the algorithm in \cite{FOR04} via a hybrid implementation which is available in \cite{Shi26}, along with our algorithm for lifting $L_p(T)\bmod p$ for genus 3 curves.

The rest of this paper is structured as follows. We provide an
 algorithm to perform the group operation on the full Jacobian in Section~\ref{section-group-naive},
 followed by a sped-up version (using \cite{FOR04}) in Section~\ref{section-group-hybrid} and its complexity analysis in Section~\ref{section-analysis}.
 Then, we explain a four-step algorithm to lift $L_p(T)\bmod p$ to $L_p(T)$ using this group operation in Sections~\ref{section-lifting}--\ref{section-lv}. We present benchmark results and timings in Section~\ref{section-implementation}.

\section*{Acknowledgments}
I am very grateful to my advisor, Andrew V. Sutherland, for suggesting the problem and for his guidance throughout this project.
I am also very grateful to Raymond van Bommel and Edgar Costa for many helpful discussions.
I also acknowledge the support of the Natural Sciences and Engineering
Research Council of Canada (NSERC).

\textbf{AI declaration:} All the writing in this paper was done by the author.
The Jacobian arithmetic, the algorithms in Steps 1 and 2, and the benchmarking were entirely designed and implemented by the author.
The algorithm in Step 3 was also designed by the author, but the implementation of Sutherland's algorithm~\cite{Sut10} was done with the help of Claude Code.
Step 4 was suggested and implemented by Claude Code, while its analysis and design were carefully developed by the author.
GPT suggested combining a budgeted version of Step 3 with Step 4 to achieve a running time of $O(p^{1/4+o(1)})$, and this method was fully verified and written out by the author.
The author also used Claude Code for minor grammar and clarity improvements.
The author assumes responsibility for all content.

\section{Naive group operation on Jacobians of smooth plane quartics}\label{section-group-naive}

In this section, we provide library functions to perform Jacobian arithmetic 
on genus $3$ curves that arise as smooth plane quartics over finite fields.
Let $C$ be a genus $3$ smooth plane quartic over $\F_q$, where $q=p^e$ is an odd prime power with $p>127$.
Our contributions are:\footnotemark
\begin{itemize}
    \item A ``naive algorithm'' capable of computing the sum of any two divisors in $\Jac(C)(\F_q)$, as well 
    as negating and hashing.
    \item A fast ``typical algorithm'' that computes the sum of two ``typical divisors'' (defined later). This 
    algorithm is described in \cite{FOR04}. See Section~\ref{section-group-hybrid}.
    \item A hybrid algorithm that uses the typical algorithm whenever possible and reverts to 
    the naive algorithm otherwise. See Section~\ref{section-group-hybrid}.
\end{itemize}
\footnotetext{Implementations of the algorithms are available in \cite{Shi26}.}
We provide two different implementations for the group operation on $\Jac(C)(\F_q)$.
Although typical addition is extremely fast, 
it does not cover all cases (for example, the identity element is not typical). In cases that are not covered,
we switch to the naive algorithm. When lifting $L$-polynomials, 
most group operations are typical, but naive cases are usually needed when $q<2^{14}$ (see Subsection~\ref{subsection-naive-vs-typical}). A hybrid implementation supports fully general group operations while utilizing the speedup whenever possible.

\subsection{Naive algorithm: preprocessing} \label{subsection-preprocess}

Let the smooth plane quartic $C$ be given by a homogeneous polynomial $g \in \F_q[x,y,z]$ of degree $4$.
We first attempt to find a tangent line $l_g(x,y,z)$ whose intersection points with $C$ are all rational.
To do this, we start with the affine patch $z\neq 0$ and iterate through the
$x$-coordinates in $\F_q$. For each $x$, we solve for the corresponding
$y$-coordinates in $\F_q$, and for each resulting point $P=(x:y:1)$ on $C$
we compute the tangent line $\ell$ at $P$ and check whether all four points of
$\ell\cdot C$ are $\F_q$-rational. We return the first such line. In practice,
such a line is found after very few iterations.

By \cite{OR10}, for $q\geq 66^2+1$, we are guaranteed to find such a tangent line. 
We ran an experiment that verified that, among all 82,240 smooth plane quartics defined over $\Q$ in Sutherland's database \cite{Sut19} and all 564 choices of $p \in [127, 66^2]$ (excluding primes of bad reduction), all but 11 curve-prime pairs have this property. 
Once such a tangent line $l_g$ is found, 
we translate $g$ and $l_g$ linearly to the standard representation $f(x,y,z)$ and $L(x,y,z)=z$, where 
the tangent point becomes $(0:1:0)$ and the other two intersection points become $(1:0:0)$ and $(x_3:y_3:0)$
for some $x_3, y_3 \in \F_q$. In this case, we define $D^\infty = P_1^\infty + P_2^\infty + P_3^\infty$, where
\[L \cdot C = \underbrace{(0:1:0)}_{P_1^\infty} + \underbrace{(0:1:0)}_{P_2^\infty} + \underbrace{(x_3:y_3:0)}_{P_3^\infty} + \underbrace{(1:0:0)}_{P_4^\infty}\]
and $L$ is the tangent line.

We perform this step only once per curve.
Afterward, this standard representation is used for all subsequent group operations on the Jacobian of this curve.

In the case that we do not find a tangent line with rational intersection points, by \cite{OR10},
when $q\geq 127$, we can still find a line with four rational intersection points $P_1^\infty, P_2^\infty, P_3^\infty, P_4^\infty$.
We define $D^\infty = P_1^\infty + P_2^\infty + P_3^\infty$. The naive group 
operation algorithm still works, but the typical group operation, described  
in Section~\ref{section-group-hybrid}, no longer applies.

\subsection{Points on the Jacobian}

Each point on $\Jac(C)(\F_q)$ can be represented as the equivalence class of a divisor
\[D=P_1+P_2+P_3-D^\infty,\]
where the points $P_i$ are defined over $\F_{q^6}$, since they belong to Galois orbits of size 1, 2, or 3,
and $\F_q$, $\F_{q^2}$, and $\F_{q^3}$ are all subfields of $\F_{q^6}$. As $D^\infty$ is fixed,
we represent $D$ by the multiset $\set{P_1,P_2,P_3}$.

We next describe the implementation of several key group operations: addition, scalar multiplication, identity, and negation. 
We also describe hashing and random point generation.

\subsection{Naive algorithm: addition and scalar multiplication}
Algorithm~\ref{alg:naiveAdd} implements point addition, and its correctness follows from the Theorem in Section~2 of \cite{FOR04}.
With addition implemented, scalar multiplication is performed using the double-and-add method.

The most difficult task in the addition algorithm is to find a cubic curve $K$
such that $\set{P_1,\ldots, P_9}\leq K\cdot C$
for nine assigned points $P_1,\ldots,P_9$. (The same applies to finding a quadratic curve $Q$
passing through five specified points with multiplicities.) 
When the multiplicities of all points are less than three, this can be achieved via linear algebra.
Otherwise, we need to use ideals.

\begin{algorithm}%
    \caption{Naive point addition}\label{alg:naiveAdd}
    \begin{algorithmic}
        \Function{NaiveAddition}{$f, D_1=D_1^+-D^\infty, D_2=D_2^+-D^\infty, P_1^\infty, P_2^\infty, P_4^\infty$}
        \State \textbf{Input:} the quartic $f$ defining $C$, two divisor classes $D_1=D_1^+-D^\infty$ and $D_2=D_2^+-D^\infty$ in $\Jac(C)(\F_q)$, and the points $P_1^\infty,P_2^\infty,P_4^\infty$ from the preprocessing step.
        \State \textbf{Output:} the sum $D_1+D_2\in\Jac(C)(\F_q)$.
        \State \textbf{if} {$D_i$ is the identity for some $i$} \Return $D_{3-i}$
        \State \Comment{Step 1: cubic interpolation}
        \State {Find a cubic $K$ such that $K\cdot C \geq {D_1^++D_2^++P_1^\infty+P_2^\infty+P_4^\infty}=:{S_1}$. To do this,}
        \If{the multiplicities of all points in $S_1$ are at most $2$}
            \State $K\gets$\Call{Lin-Alg-interpolation}{$f, S_1$}, see Algorithm~\ref{alg:lin-alg-interpolation}
        \Else 
            \State $K\gets$\Call{Ideal-interpolation}{$f,S_1$}, see Algorithm~\ref{alg:ideal-interpolation}
        \EndIf
        \State \Comment{Step 2: cubic intersection}
        \State Compute $D_3^+$ such that $K\cdot C = S_1+D_3^+$. In {Magma}, use \Call{IntersectionNumbers}{}.
        \State \Comment{Step 3: quadratic interpolation}
        \State Find a quadratic $Q$ such that $Q\cdot C \geq {D_3^++P_1^\infty + P_2^\infty}=:{S_2}$. To do this,
        \If{the multiplicities of all points in $S_2$ are at most $2$}
            \State $Q\gets$\Call{Lin-Alg-interpolation}{$f, S_2$}, see Algorithm~\ref{alg:lin-alg-interpolation}
        \Else 
            \State $Q\gets$\Call{Ideal-interpolation}{$f,S_2$}, see Algorithm~\ref{alg:ideal-interpolation}
        \EndIf
        \State \Comment{Step 4: quadratic intersection}
        \State Compute $D_4^+$ such that $Q\cdot C = S_2+D_4^+$. In {Magma}, use \Call{IntersectionNumbers}{}.
        \State \Return $D_4=D_4^+-D^\infty$

        \EndFunction

    \end{algorithmic}
\end{algorithm}

\begin{algorithm}%
    \caption{Linear algebra interpolation}\label{alg:lin-alg-interpolation}
    \begin{algorithmic}   
        \Function{Lin-Alg-interpolation}{$f, S$}
        \State \textbf{Input:} the quartic $f$ defining $C$, and a multiset $S$ of points of $C$ with $\abs{S}\in\set{5,9}$, each of multiplicity at most $2$.
        \State \textbf{Output:} a nonzero homogeneous polynomial of degree $2$ (if $\abs{S}=5$) or degree $3$ (if $\abs{S}=9$) vanishing at the points of $S$ to the prescribed multiplicities.
         \State \Comment{Step 1: Define monomial basis $M$}
        \If{$\abs{S}=9$}
            \State $M\gets [x^3, y^3, z^3, x^2y, xy^2, x^2z, z^2x, y^2z, z^2y, xyz]$
        \Else \ we have {$\abs{S}=5$}
            \State $M\gets [x^2,y^2,z^2,xy,xz,yz]$
        \EndIf
         \State \Comment{Step 2: Impose constraints on points}
        \State Set up a linear system where $\sum_{i=1}^{\abs{M}} \lambda_i M_i=0$
        \For{$P=(x,y,z)\in S$}
            \State Impose the condition  $\sum_{i=1}^{\abs{M}} \lambda_i M_i(x,y,z)=0$
        \EndFor
         \State \Comment{Step 3: Impose constraints on tangents}
        \For{$P=(x,y,z)\in S$ with multiplicity $2$}
            \State Let $\alpha x+\beta y+\gamma z=0$ be the tangent line to $C$ at $P$
            \State Impose the following conditions:
            \State $\sum_{i=1}^{\abs{M}} \lambda_i\left(\tfrac{\partial M_i}{\partial x}(x,y,z)\beta - \tfrac{\partial M_i}{\partial y}(x,y,z)\alpha \right)=0$
            \State $\sum_{i=1}^{\abs{M}} \lambda_i\left(\tfrac{\partial M_i}{\partial y}(x,y,z)\gamma - \tfrac{\partial M_i}{\partial z}(x,y,z)\beta \right)=0$
            \State $\sum_{i=1}^{\abs{M}} \lambda_i\left(\tfrac{\partial M_i}{\partial z}(x,y,z)\alpha - \tfrac{\partial M_i}{\partial x}(x,y,z)\gamma \right)=0$
        \EndFor

        \State Solve the system in $\lambda_i$
        \State \Return $\sum_{i=1}^{\abs{M}} \lambda_i M_i$
        \EndFunction
    \end{algorithmic}
\end{algorithm}

\begin{algorithm}%
    \caption{Ideal interpolation}\label{alg:ideal-interpolation}
    \begin{algorithmic}   
        \Function{Ideal-interpolation}{$f,S$}
        \State \textbf{Input:} the quartic $f$ defining $C$, and a multiset $S$ of points of $C$ with $\abs{S}\in\set{5,9}$.
        \State \textbf{Output:} as in Algorithm~\ref{alg:lin-alg-interpolation}, a nonzero homogeneous polynomial of degree $2$ (if $\abs{S}=5$) or degree $3$ (if $\abs{S}=9$) vanishing at the points of $S$ to the prescribed multiplicities.
        \State Apply a linear transformation $M\in\GL_3(\F_q)$ such that all points in $S$ are on the affine patch $z\neq0$.
        \State Let $g(x,y,z):=f(Mx,My,Mz)$. Transform the points in $S$ as well.
        \State Dehomogenize $g$ by setting $z=1$. Call this $g_z$.
        \State Let the distinct points in $S$ be $P_i=[x_i,y_i]$ for $1\leq i\leq n$, and let $m_i$ be the multiplicity of $P_i$.
        \State Compute the ideal $I = \langle g_z\rangle + \bigcap_{i=1}^{n}\langle X-x_i, Y-y_i\rangle^{m_i}\subseteq\F_{q^6}[X,Y]$. 
        \State Let $d=2$ if $\abs{S}=5$ and $d=3$ if $\abs{S}=9$. Let $h_z\in I\cap\F_q[X,Y]$ be a nonzero polynomial of total degree at most $d$.
        \State Homogenize $h_z$ to degree $d$ with respect to $z$, i.e., $h(x,y,z):=z^d\,h_z(x/z,y/z)$.
        \State \Return $h(M^{-1}x, M^{-1}y, M^{-1}z)$
        \EndFunction
    \end{algorithmic}
\end{algorithm}
The constraints in Step 3 of Algorithm~\ref{alg:lin-alg-interpolation} ensure the following tangency condition at each point $P\in S$ of multiplicity $2$: $\alpha x+\beta y+\gamma z=0$ and $(\partial f/\partial x)x + (\partial f/\partial y)y + (\partial f/\partial z)z=0$.

Experimental results indicate that the ideal interpolation method is approximately twice as slow as the linear algebra approach.
Therefore, we use the linear algebra method whenever possible.

\subsection{Naive algorithm: identity and inverse}
The identity element is given by $ P_1^\infty+P_2^\infty+P_3^\infty - D^\infty$.
\begin{algorithm}%
    \caption{Divisor inversion}\label{alg:naive-inverse}
    \begin{algorithmic}   
        \Function{Divisor inversion}{$f,D=D^+-D^\infty$}
        \State \textbf{Input:} the quartic $f$ defining $C$, and a divisor class $D=D^+-D^\infty\in\Jac(C)(\F_q)$.
        \State \textbf{Output:} the negation $D^{-1}=-D\in\Jac(C)(\F_q)$.
        \State Find the quadratic $Q$ such that $Q\cdot C \geq D^++2P_4^\infty =: S$. To do this,
        \If{the multiplicities of all points in $S$ are at most $2$}
            \State $Q\gets$\Call{Lin-Alg-interpolation}{$f, S$}
        \Else 
            \State $Q\gets$\Call{Ideal-interpolation}{$f,S$}
        \EndIf
        \State Compute $(D^{-1})^+$, where $Q\cdot C = S + (D^{-1})^+$. In Magma, use \Call{IntersectionNumbers}{}.
        \State \Return $D^{-1}=(D^{-1})^+-D^\infty$
        \EndFunction
    \end{algorithmic}
\end{algorithm}

\begin{lemma} Algorithm~\ref{alg:naive-inverse} correctly computes the negation of a divisor.
\end{lemma}

\begin{proof}
    Consider the following:
    \begin{align*}
        D+D^{-1} &= (D^+-D^\infty) + ((D^{-1})^+-D^\infty) \\
        &=(D^+-D^\infty) + ((D^{-1})^+-D^\infty) + 2P_4^\infty - 2P_4^\infty \\
        &=(Q\cdot C) -2P_4^\infty - 2D^\infty = (Q\cdot C) - 2(P_4^\infty +P_1^\infty+P_2^\infty+P_3^\infty) \\
        &= (Q\cdot C) - 2 L\cdot C,
    \end{align*}
    where $Q\cdot C - 2L\cdot C$ is linearly equivalent to 0.
\end{proof}
\subsection{Naive algorithm: hash function}

To perform the baby-step giant-step algorithm 
in Section~\ref{section-lifting}, we need a hash function 
for the group $\Jac(C)$ that uniquely identifies group elements (a perfect hash). If the hash of $D=D^+-D^\infty$ is defined as the multiset of the 
three points in $D^+$, it is not unique. There exist situations 
where $P_1+P_2+P_3- D^\infty \sim P_1'+P_2'+P_3'- D^\infty$
but the multisets $\set{P_1,P_2,P_3}$ and $\set{P_1',P_2',P_3'}$
are different. 
The following proposition determines exactly when 
this representation is not unique, and it describes how to generate a hash in those situations.
\begin{proposition}\label{Proposition-hash}
    Write $D=D^+-D^\infty$, where $D^+=P_1+P_2+P_3$.
    Exactly one of the following holds:
    \begin{enumerate}
        \item $D$ can be represented uniquely in the form $D^+-D^\infty$. This happens when $P_1,P_2,P_3$
        are not collinear.
        In this case, we define $\Hash(D)$ to be the multiset $\set{P_1,P_2,P_3}$.
        \item $D$ is linearly equivalent to a divisor of the form $P_4^\infty-P$.
         This happens when $P_1,P_2,P_3$ are 
        collinear on a line $\ell$ and $\ell\cdot C = P_1+P_2+P_3+P$.
        In this case, we define $\Hash(D)$ to be the set $\set{P}$.
    \end{enumerate}
\end{proposition}

\begin{proof}
    The divisor $D^+=P_1+P_2+P_3$ has degree $3$, and $C$ has genus $3$, so by 
    Riemann--Roch, we have 
    \[\ell(D^+)=\ell(K_C-D^+)+1,\]
    where $K_C$ is the canonical divisor. Note that since $K_C$ has degree $4$, the divisor $K_C-D^+$
    has degree $1$. On a non-hyperelliptic curve of genus $3$, the dimension of
    the space of global sections of a degree $1$ divisor is at most $1$, so
    $\ell(K_C-D^+)\leq 1$.

    When $\ell(K_C-D^+)=0$, we have $\ell(D^+)=1$, so $D^+$ is the unique effective divisor in its equivalence class, and the multiset $\set{P_1,P_2,P_3}$ determines $D$.

    On the other hand, when $\ell(K_C-D^+)=1$, the divisor $K_C-D^+$ has a unique effective divisor in its class, and we denote it by $P$. 
    We have $K_C-D^+\sim P$, or equivalently, $K_C\sim D^++P$. Since $C$ is 
    a smooth plane quartic, its effective canonical divisors are the line sections, so $D^++P$ is a line section. In particular,
    the points in $D^+$ are collinear. 

    The above argument shows that exactly one of the two conditions in the proposition holds.

    It remains to show that in the second case, if $P\neq Q$, then $\set{P}$ and $\set{Q}$ represent different classes.

    Let $P_1, P_2, P_3$ and $P$ be collinear and $D^+=P_1+P_2+P_3$. Let $\ell$ be the line through $P_1,P_2,P_3,P$. Note that $\ell\cdot C\sim L\cdot C$, 
    where $L$ is the tangent line from the preprocessing step.

    Then, \[D = D^+-D^\infty = D^++P-D^\infty-P = \ell\cdot C - D^\infty - P \sim L\cdot C - D^\infty - P = P^\infty_4-P.\]
    Therefore, if both $P$ and $Q$ represent the same class, we would have $P\sim Q$. On a curve of genus $\geq 1$, 
    this implies $P=Q$. This concludes the proof.

\end{proof}

\subsection{Naive algorithm: point generation}
Every effective $\F_q$-rational divisor of degree $3$ is a sum of Galois orbits whose 
sizes add to $3$. There are 
three cases in which this can happen, each with a different probability.
The following generates a random point on $\Jac(C)(\F_q)$:
\begin{itemize}
    \item With probability $\frac{1}{2}$, we return $D=\set{P_1,P_1',P_2}$ where $P_1$ is a random point over $\F_{q^2}\setminus\F_q$ with
    Galois conjugate $P_1'$, and $P_2$ is a random point over $\F_q$.
    \item With probability $\frac{1}{3}$, we return $D=\set{P_1, P_1', P_1''}$
    where $P_1$ is a random point over $\F_{q^3}\setminus\F_q$ with
    Galois conjugates $P_1'$ and $P_1''$.
    \item With probability $\frac{1}{6}$, we select three random points over $\F_q$ and return $D=\set{P_1,P_2,P_3}$.
\end{itemize}
From the Hasse--Weil bound $\#C(\F_{q^k}) = q^k + O(q^{k/2})$ and a standard counting argument,
the three types above contain approximately $q^3/2$, $q^3/3$, and $q^3/6$ divisors, respectively.
Sampling each with the corresponding probability therefore yields a point that is asymptotically uniform on $\Jac(C)(\F_q)$.

\section{Using typical addition to speed up the naive group operation} \label{section-group-hybrid}

For the purpose of lifting $L$-polynomials (see Section~\ref{section-lifting}), 
although it is possible to use only the naive group addition (just introduced in the previous section),
it would be very slow. See Subsection~\ref{subsection-naive-vs-typical} for a speed comparison. 
 We now introduce typical 
addition (by Flon, Oyono, and Ritzenthaler in \cite{FOR04}) in Subsection~\ref{Subsection-Ritz},
and explain how it can speed up the lifting process significantly
in Subsection~\ref{Subsection-hybrid}.

\subsection{Flon, Oyono, and Ritzenthaler's typical divisor addition algorithm}\label{Subsection-Ritz}

We assume that the preprocessing step is done for the smooth plane quartic (see Subsection~\ref{subsection-preprocess}).
We define what it means for a divisor to be typical and atypical.
\begin{definition}
    Let $D=P_1+P_2+P_3-D^\infty$. We define $D$ to be typical if it satisfies all of the following:
    \begin{itemize}
        \item $P_1,P_2,P_3$ are not collinear.
        \item All points lie on the affine patch $z \neq 0$, so normalizing at $z=1$,
        we may write $P_i=(x_i:y_i:1)$.
        \item The affine coordinates $x_1,x_2,x_3$ are pairwise distinct.
    \end{itemize}
    Otherwise, we call the divisor atypical.
\end{definition}
When a divisor is typical, it is represented by the Mumford representation $(u(X), v(X))$,
where $u(X)$ and $v(X)$ are polynomials in $\F_q[X]$ of degrees $3$ and $2$, respectively. 
The roots of $u(X)$
are $x_1,x_2,x_3$ and  $v(x_i)=y_i$ for each $i$.

Flon, Oyono, and Ritzenthaler \cite{FOR04} provide an algorithm to compute the sum of two typical divisors by working with their Mumford representations (see the Algorithm~1 pseudocode in \cite{FOR04}).
The main idea is the same as in Algorithm~\ref{alg:naiveAdd}, which is to first find a cubic curve through the points, 
compute the residual divisor, and repeat the process by finding a quadratic curve and computing its residual divisor.

Although \cite{FOR04} provides a Magma program for typical divisor addition, it is specialized for curves with a flex (a tangent line that intersects
the curve at one point with multiplicity $3$).  
Conjecturally, $1-e^{-1}\approx 63\%$ of smooth plane quartics over $\F_q$ have a rational flex. See \cite{OR10}.
This limitation prevents the use of their Magma program for the general case.
 
To accommodate smooth plane quartics with a tangent (not necessarily a flex), we provide an independent implementation of Algorithm~1 of \cite{FOR04} in Magma (see \cite{Shi26}).

\subsection{Hybrid addition}\label{Subsection-hybrid}

Although typical addition is very fast and covers most cases in the context of lifting $L$-polynomials,
it may fail in two situations:
\begin{itemize}
    \item One or both addends are not typical divisors (e.g.\ one is the identity, or
    the support points are collinear or not on the affine patch $z \neq 0$).
    \item Both addends are typical, but typical addition throws an error --- for instance, the
    Step~1 linear system has no solution, $\omega_1$ and $u_1$ are not coprime when computing
    $t_1=\omega_1^{-1}\bmod u_1$, the output $(u,v)$ has the wrong degree, or $\gcd(u,u')\neq 1$.
    We catch these and raise an exception.
\end{itemize}
In these cases, we fall back to the naive addition algorithm and convert the result back to typical
representation whenever possible. See Algorithm~\ref{alg:hybrid-add} for details.

\begin{remark}
    To ensure consistency, our implementation strictly enforces the representation format: 
    typical divisors are always stored in Mumford representation $(u(X), v(X))$, 
    while atypical divisors are stored as multisets of three points.
    \end{remark}

\begin{algorithm}%
    \caption{Hybrid addition algorithm}\label{alg:hybrid-add}
    \begin{algorithmic}   
        \Function{Addition}{$f, D_1,D_2$}
        \State \textbf{Input:} the quartic $f$ defining $C$, and two divisor classes $D_1,D_2\in\Jac(C)(\F_q)$, each stored as a typical or atypical divisor.
        \State \textbf{Output:} the sum $D_1+D_2\in\Jac(C)(\F_q)$, stored as a typical or atypical divisor.
        \If{$D_i$ is the identity} \Return $D_{3-i}$\EndIf
        \If{both $D_1$ and $D_2$ are typical, with $D_1=(u_1,v_1)$, $D_2=(u_2,v_2)$} 
        \State \textbf{Try} $D_3\gets$ \Call{TypicalAddition}{$f,D_1,D_2$}
        \State Given $D_3=(u_3,v_3)$, assert that $(\deg(u_3), \deg(v_3))=(3,2)$ and $\gcd(u_3,u_3')=1$
        \State \Return $D_3$ as a typical divisor represented by $(u_3,v_3)$
        \State \textbf{Catch} Error $e$.
        \EndIf
        \If{$D_1$ is typical} \Comment{attempted typical add fails, so we use naive add}
        \State Represent $D_1$ as $\set{P_1,P_2,P_3}$, with $P_i=(x_i,y_i)$, where $x_i$ are roots of $u_1$, $y_i=v_1(x_i)$
        \EndIf
        \If{$D_2$ is typical}
        \State Represent $D_2$ as $\set{Q_1,Q_2,Q_3}$ with $Q_i=(x_i,y_i)$, where $x_i$ are roots of $u_2$, $y_i=v_2(x_i)$
        \EndIf

        \State $D_3=\set{R_1,R_2,R_3} \gets$ \Call{NaiveAddition}{$f,D_1,D_2$}.
        \If{$\set{R_1,R_2,R_3}$ satisfies the conditions for a typical divisor}
        \State Compute $(u_3,v_3)$ representing $D_3$
        \State\Return $D_3$ as a typical divisor $(u_3,v_3)$.
        \Else 
        \State\Return $D_3$ as an atypical divisor represented by $\set{R_1,R_2,R_3}$
        \EndIf

        \EndFunction
    \end{algorithmic}
\end{algorithm}

When $D$ is a typical divisor, it is uniquely represented by $D=(u,v)$, so its hash is given by $(u,v)$.
When $D$ is atypical, we use $\Hash(D)$ given by Proposition~\ref{Proposition-hash}.
This guarantees a canonical representation for each divisor. Thus, hashing is unambiguous.

\subsection{Naive additions vs.\ typical additions} \label{subsection-naive-vs-typical}

We compare the speed of naive and typical addition, as well as their frequency. We only 
record cases where neither addend is the identity. 

Experimental results for $p \in [2^9, 2^{20}]$ show that typical addition 
is between 50 and 200 times faster than naive addition.

In the context of lifting $L$-polynomials, we observe that as $p$ increases, the ratio of naive additions
to typical additions decreases: from roughly 1/10 for $p < 2^{10}$ to 1/100 for $p \in [2^{10}, 2^{13}]$, and further for larger $p$.

We ran an experiment 
that, for each $n$, lifts 1,000 random instances 
of $L_p(T)\bmod p$ for various $n$-bit primes $p$, and records how many instances invoked naive addition at least once.
See Table~\ref{tab:naive-needs}. 
\begin{table}[h]
\centering\small
\setlength{\tabcolsep}{1.5pt}
\begin{tabular}{|l|c|c|c|c|c|c|c|c|c|c|c|c|c|c|c|c|c|c|c|c|c|}
\hline
$n$ & 10 & 11 & 12 & 13 & 14 & 15 & 16 & 17 & 18 & 19 & 20 & 21 & 22 & 23 & 24 & 25 & 26 & 27 & 28 & 29 & 30 \\ \hline
    & 995 & 941 & 822 & 616 & 432 & 300 & 171 & 114 & 62 & 46 & 26 & 12 & 8 & 9 & 2 & 2 & 1 & 2 & 1 & 0 & 4 \\ \hline
\end{tabular}
\bigskip
\caption{Number of instances of lifting $L_p(T)\bmod p$, $p\in[2^{n-1},2^{n})$, that required naive addition at least once, among 1,000 random instances}
\label{tab:naive-needs}
\end{table}
The results demonstrate that while naive addition is rarely required, 
its implementation remains essential. We also found cases for $\Jac(C)(\F_p)$ with a $25$-bit prime $p$ where 
there exist non-identity atypical points that are multiples of a typical point.

\section{Complexity analysis for the group algorithms} \label{section-analysis}

We now analyze the complexity of the group operations. We compute bit complexity
and we use $n = \log q$ to estimate the number of bits needed to represent an element of $\F_q$. Let $\mathsf{M}(n)$
denote the cost of field multiplication, which, by \cite{Har21}, is $O(n\log n)$. Throughout the 
analysis, we note that we are working with polynomials of degree up to $4$, a constant number of variables, 
and matrices of dimension at most $9 \times 9$.

An instance of typical addition in $\Jac(C)(\F_{q})$ requires a fixed number of field operations, matrix multiplications, 
matrix inversions, resultant computations, and greatest common divisor computations for polynomials of degree up to $3$. The dominating operations are the matrix inversions and 
greatest common divisor computations, so the total time complexity is $O(\mathsf{M}(n)\log n)= O(n\log^2n)$.

Computing naive addition and negation is more expensive. The cubic and quadratic interpolation steps consist of 
either a linear-algebra step, which requires a fixed number of field operations and matrix inversions, or an 
ideal-based calculation, which requires a fixed number of field operations and matrix multiplications, as well as ideal 
intersections. Ideal intersection can be computed using Faug\`ere's F4 algorithm \cite{Fau99}. Computing the 
intersection divisor relies on resultants and root-finding, with the latter being the dominating factor.
 The resultant can be computed in a fixed number of field operations since 
the degrees are at most $4$, while root-finding is $O(\mathsf{M}(n^2))$ by the Cantor--Zassenhaus algorithm \cite{CZ81}. 
 Therefore, the total time complexity is $O(\mathsf{M}(n^2))= O(n^2\log n)$.

Converting from typical to atypical representation is dominated by root-finding for a degree $3$ polynomial, 
which is $O(\mathsf{M}(n^2))$ via the Cantor--Zassenhaus algorithm \cite{CZ81}. Converting from atypical to typical representation can be done 
using Lagrange interpolation, which costs $O(\mathsf{M}(n\log n))$. 

When the divisor is typical, or when it is atypical but the three points 
are not collinear, the hash is simply the representation of the points, which takes $O(1)$ time. Otherwise, 
we need to compute the fourth intersection point of $C$ and the line $\ell$ through the three points.
This can be done by computing the intersection divisor, which takes $O(\mathsf{M}(n^2)) = O(n^2\log n)$ time.

\section{Lifting $L$-polynomials of genus $3$ curves} \label{section-lifting}

Let $C/\F_p$ be a genus $3$ curve which is either a smooth plane quartic curve or a hyperelliptic curve.
In the average polynomial-time setting, this will be the reduction of a genus $3$ curve over $\Q$ 
with good reduction at $p$.
In this section, we provide an algorithm that takes $C$ and $L_p(T) \bmod p$ as inputs and 
outputs the integer polynomial $L_p(T)$ in  $O(p^{1/4+o(1)})$  expected time.
 
More precisely, the $L$-polynomial is of the form 
\[L_p(T) = p^3T^6 + p^2a_1T^5 + pa_2T^4 + a_3T^3 + a_2T^2 + a_1T + 1, \t{ where }a_1,a_2,a_3\in\Z,\]
and the algorithms in \cite{Sut20,CHS23} provide the values of $a_1\bmod p$, $a_2\bmod p$, and $a_3\bmod p$. Our algorithm 
takes these as inputs and outputs $a_1,a_2,a_3$.

Our algorithm relies on two key ingredients.
\begin{itemize}
    \item \textbf{Ingredient 1}: Let $\Jac(C)(\F_{p^k})$ denote the Jacobian of $C$ over $\F_{p^k}$. 
    Then, we can deduce the group order of $\Jac(C)(\F_{p^k})$ from the $L$-polynomial using the following result:
    \[\#\Jac(C)(\F_{p^k})=\prod_{\zeta^k=1} L_p(\zeta).\]
    We will use the equality $\#\Jac(C)(\F_p) = L_p(1)$ very often.
    A proof of this fact 
    can be found in Chapter~8 of \cite{Lor96}.
    Using this fact, we can also deduce the order of the trace zero variety:
    \[\#(\Jac(C)(\F_{p^k})/\Jac(C)(\F_p)) = \prod_{\zeta^k=1, \zeta\neq 1} L_p(\zeta).\]

    \item \textbf{Ingredient 2}: We assume access to the following functionality:
    
        \begin{itemize}
            \item Group functionalities on $\Jac(C)(\F_{p^k})$ for $k\geq1$. These include addition, negation, the identity element, 
            and random element generation.
            \item A perfect hash function for group elements of $\Jac(C)(\F_{p^k})$ for $k\geq 1$.
            \item A perfect hash function for group elements of the trace zero variety $\Jac(C)(\F_{p^2})/\Jac(C)(\F_p)$. This functionality is 
            needed in order to perform the baby-step giant-step algorithm on the trace zero variety.

            When $C$ is a hyperelliptic curve,
            \[
                \Jac(C)(\F_{p^2})/\Jac(C)(\F_p)
                \simeq \Jac(\widetilde C)(\F_p),
            \]
            where $\widetilde C$ is the quadratic twist of $C$. 
                
            When $C$ is a smooth plane quartic, we can obtain a canonical realization of the trace zero variety 
            inside $\Jac(C)(\F_{p^2})$.
            Let $F$ denote the $p$-power Frobenius map and define
            \[
                \phi:\Jac(C)(\F_{p^2})\longrightarrow\Jac(C)(\F_{p^2}),
                 \quad D\mapsto F(D)-D.
            \]
            Since the kernel of the morphism is $\Jac(C)(\F_p)$, the first isomorphism theorem gives \[\Jac(C)(\F_{p^2})/\Jac(C)(\F_p)\simeq \operatorname{Im}(F-1).\]
            
            More explicitly, for $D,D'\in\Jac(C)(\F_{p^2})$,
            \[
            \begin{aligned}
                \phi(D)=\phi(D')&\iff F(D-D')-(D-D')=0\\
                \iff D-D'\in\ker(F-1)&\iff D-D'\in\Jac(C)(\F_p).
            \end{aligned}
            \]
        
            Therefore, two elements in $\Jac(C)(\F_{p^2})$ have the same image under $\phi$ if and only if they 
            represent the same coset. Hence, we can hash a coset by hashing its image under $\phi$ using the perfect hash function
            on $\Jac(C)(\F_{p^2})$.

            Moreover, we may generate random elements of the trace zero variety by choosing a random element 
            $D\in\Jac(C)(\F_{p^2})$ and computing $\phi(D)$.
            \end{itemize} 
            When $C$ is a smooth plane quartic, these functionalities are fully provided by Section~\ref{section-group-naive}
            and Section~\ref{section-group-hybrid}, with a Magma implementation available at \cite{Shi26}.
            When $C$ is a hyperelliptic curve, we can use Cantor's algorithm \cite{Can87} (built into Magma).
            For a more efficient implementation, see \cite{Sut19b}.
\end{itemize}

\subsection{Overview of the algorithm}

Using the ingredients above, we determine the coefficients of the $L$-polynomial, $(a_1, a_2, a_3)$, in four steps:
\begin{itemize}
    \item \textbf{Step 1:} Using Hasse--Weil bounds and bounds in \cite{KS08}, 
    we determine a finite list of candidate triples.
    \item \textbf{Step 2:} Using the fact that the group order must be a multiple of the group exponent,
    we apply a baby-step giant-step algorithm to eliminate incompatible candidates via group operations on $\Jac(C)(\F_p)$. This step is itself probabilistic, relying on random divisor generation and probabilistic root-finding.
    When $p>1600$, this uniquely determines the $L$-polynomial in all but approximately $0.009\%$ of cases, an empirical outcome from our experiments (see Remark~\ref{experiment-step3examples}).
    \item \textbf{Step 3:} In the rare cases where multiple candidates remain, we repeatedly choose a prime $\ell$ and an integer $b\geq 0$
    strategically, attempt to construct a subgroup of size $\ell^{b+1}$, and eliminate all the
    candidates whose Jacobian sizes are not divisible by $\ell^{b+1}$. We allocate a budget to this step.
    \item \textbf{Step 4:} If Step 3 cannot single out the correct candidate within the allocated budget,
     we compute $L_p(T)\bmod 2$ and $L_p(T) \bmod 3$
    using our Jacobian arithmetic over field extensions of $\F_p$ that realize the full $2$- or $3$-torsion. One can prove that 
    Step 4 terminates and yields a unique candidate.
\end{itemize}

\subsection{Step 1}\label{section-step1}

The Hasse--Weil bounds imply:
\[-6\sqrt{p}\leq a_1\leq 6\sqrt{p},\ -15p\leq a_2\leq 15 p,\ -20 p^{3/2}\leq a_3 \leq 20p^{3/2}.\]

If $p>144$, then $12\sqrt{p}<p$. Hence the width of the interval $[-6\sqrt{p},6\sqrt{p}]$
is smaller than $p$ and yields a unique $a_1$. 
Throughout the rest of the paper, we assume $p>144$ and that $a_1$ is unique.
Furthermore, we have the following lemma:
\begin{lemma}
    There are at most 6 choices for $a_2$.
\end{lemma}
\begin{proof}
    Proposition~4 in \cite{KS08} determines bounds for $a_2$ given $p$ and $a_1$.
    Define $\delta$ to be the distance from $a_1/\sqrt{p}$ to its nearest integer equivalent to $0\bmod 4$.
    Then, the lower bound and the upper bound for $a_2$ are given by:
    \[B_{\t{low}} = -p+\frac{a_1^2-p\delta^2}{2} \t{ and } B_{\t{high}} =3p+\frac{a_1^2}{3}.\] 

    The difference between the upper and lower bounds is the following:
    \begin{align*}
        &3p+\frac{a_1^2}{3} - \left(-p+\frac{a_1^2-p\delta^2}{2}\right)
        =4p-\frac{1}{6}a_1^2+\frac12 p\delta^2
        \leq 4p+2p=6p,
    \end{align*}
    where the last inequality follows from $\delta\leq 2$.
    Equality is obtained $\iff$ 
    $\delta^2=4$ and $a_1=0$, which is impossible: $\delta=\pm2$
    $\iff a_1/\sqrt{p}$ is an integer $\equiv 2\bmod4$, contradicting that $a_1=0$.
    Therefore, the inequality is strict and there are at most 6 choices for $a_2$.
\end{proof}
Let $a_{3,\min}$ and $a_{3,\max}$ be the smallest and largest integers in $[-20p^{3/2},20p^{3/2}]$
with the correct residue modulo $p$, respectively. Then, the list of $a_3$ values is given by $[a_{3,\min}, a_{3,\min}+p, \ldots, a_{3,\max}]$.
This list is independent of the candidate $a_2$ values.

To summarize, the output of this step is a list of possible $(a_1,a_2,a_3)$ triples, where $a_1$ is unique,
$a_2$ can be chosen from at most 6 values, and $a_3$ can be chosen from at most $40\sqrt{p}$ values. 

\section{Step 2: Baby-step giant-step algorithm}

The second step narrows down the list of possible $(a_1,a_2,a_3)$ triples from Step 1. Often, the second step uniquely 
determines the correct candidate and the algorithm terminates here.

For a candidate triple $(a_1,a_2,a_3)$, let $L_{p,a_1,a_2,a_3}(T):= p^3T^6 + a_1p^2T^5 + a_2pT^4 + a_3T^3 + a_2T^2+a_1T+1$. The main idea is to generate a random element $D$ in $\Jac(C)(\F_p)$ and for each candidate triple 
$(a_1,a_2,a_3)$, using the Jacobian group operation, determine whether $L_{p,a_1,a_2,a_3}(1)$ (which is $\#\Jac(C)(\F_p)$ if 
the candidate is correct) annihilates $D$. If not, $(a_1,a_2,a_3)$ must be a wrong guess and we eliminate it.

Once we generate a random element $D\in \Jac(C)(\F_p)$, if 
we compute $L_{p,a_1,a_2,a_3}(1)\cdot D$ for each of the $O(\sqrt{p})$ candidate triples $(a_1,a_2,a_3)$ naively,
it would take $O(\sqrt{p}\log p)$ group operations. To speed this up, 
we use a baby-step giant-step (BSGS) algorithm (see Algorithm~\ref{alg:step-p-BSGS}), which only takes $O(p^{1/4}\log p)$ group operations. 
We also perform this operation on the trace zero variety $\Jac(C)(\F_{p^2})/\Jac(C)(\F_p)$. We next describe 
this algorithm in detail, prove its correctness, and analyze its complexity.

\subsection{The algorithm}

Algorithm~\ref{alg:BSGS} works as a harness that combines the results of the baby-step giant-step 
algorithm on both $\Jac(C)(\F_p)$ and $\Jac(C)(\F_{p^2})/\Jac(C)(\F_p)$. The algorithm first 
tries the baby-step giant-step algorithm on $\Jac(C)(\F_p)$. It is often the case that 
this step alone determines the unique candidate. Empirically, the algorithm finishes here and returns in $99.991\%$ of cases. 
If multiple candidates remain, then the baby-step giant-step algorithm is performed on 
 $\Jac(C)(\F_{p^2})/\Jac(C)(\F_p)$ and we take the intersection of the surviving candidates.
The above procedure is repeated until the number 
of remaining candidates is at most $\loopsize$. Lemma~\ref{lem:step2-loopsize}
guarantees the termination of this algorithm.

The following proposition shows us when distinct candidate triples yield distinct Jacobian group orders:
\begin{proposition}
    When $p>1600$, distinct pairs $(a_2,a_3)$ yield distinct group orders.  
    Hence $\#\Jac(C)(\F_p)$ uniquely identifies the triple $(a_1,a_2,a_3)$.
\end{proposition}

\begin{proof}
    Let the list of $a_2$ candidates be $a_{2,1},a_{2,2},\ldots, a_{2,k}$,
    where $a_{2,i+1}=a_{2,i}+p$ and $k\leq 6$.
    Recall that the size of the Jacobian is given by $\#\Jac(C)(\F_p)=p^3+1+a_1(p^2+1)+a_2(p+1)+a_3$.

    Let $B:=p^3+1+a_1(p^2+1)+(a_{2,1}-p)(p+1)$.
    For each candidate $a_{2,i}$, $1\leq i\leq k$, 
    the possible values of $a_3$ yield an interval of possible group orders \[I_i := [B+ip(p+1)-20p^{3/2},\ B+ip(p+1)+20p^{3/2}].\]
    For these intervals to be disjoint, one needs
        $B+ip(p+1)+20p^{3/2} < B+(i+1)p(p+1)-20p^{3/2}$,
    which is achieved when $40p^{1/2}<p+1$. This is true when $p>1600$.
\end{proof}

When $p<1600$, it is possible that distinct $(a_2,a_3)$ pairs yield the same group order. 
In this case, if running Algorithm~\ref{alg:BSGS} does not yield a unique triple $(a_1,a_2,a_3)$,
it still narrows down the candidates to a small list and
we can determine $L_p(T)$ directly by counting the points of $C$ over $\F_p$, $\F_{p^2}$, and $\F_{p^3}$. For $p>1600$,
determining $L_p(T)$ reduces to determining the group order of $\Jac(C)(\F_p)$.

\begin{algorithm}
    \caption{Step 2: Baby-step giant-step harness}\label{alg:BSGS}
    \begin{algorithmic}
        \State \textbf{Input:} $\Jac(C)(\F_p), \Jac(C)(\F_{p^2})$ and $(a_1\bmod p,a_2\bmod p,a_3\bmod p)$
        \State \textbf{Output:} the triple $(a_1,a_2,a_3)$ if a unique candidate remains, and otherwise a reduced list of candidate triples, passed to Step 3.
        \State CandidateTriples $\gets$ all candidates from Step 1

        \RepeatUntil{\#CandidateTriples $\leq \loopsize$}
        \State $S_1\gets$ \Call{Step-$p$ BSGS}{$\Jac(C)(\F_p),1,\text{CandidateTriples}$}
        \If{$\#S_1=1$}
            \State \Return the unique element of $S_1$
        \EndIf
        \State $S_2\gets$ \Call{Step-$p$ BSGS}{$\Jac(C)(\F_{p^2})/\Jac(C)(\F_p),-1,\text{CandidateTriples}$}
        \State CandidateTriples $\gets S_1 \cap S_2$

        \EndRepeat
        \State Proceed to Step 3
    \end{algorithmic}
    
\end{algorithm}

\begin{algorithm}%
    \caption{Step-$p$ baby-step giant-step elimination}\label{alg:step-p-BSGS}
    \begin{algorithmic}
        \Function{Step-$p$ BSGS}{$G,\epsilon,\text{CandidateTriples}$}
        \State \textbf{Input:} a finite abelian group $G$, $\epsilon\in\{1,-1\}$, and a list of candidate triples
        \State \textbf{Output:} the candidate triples compatible with a random element of $G$
        \State $r\gets \lceil\sqrt{40}p^{1/4}\rceil$,
        $s\gets \lceil\sqrt{40}p^{1/4}\rceil$
        \Comment{numbers of baby and giant steps}
        \State CandidateTriplesTmp $\gets [\ ]$
        \For{$a_2$ occurring in CandidateTriples}

            \State $D\gets$ RandomElement($G$)
            \State $L_{\mathrm{fixed}}
            \gets (p^3+1)+\epsilon a_1(p^2+1)+a_2(p+1)$
            \If{$\epsilon=1$}
                \State $a_{3,\mathrm{start}}\gets a_{3,\min}$
            \Else
                \State $a_{3,\mathrm{start}}\gets a_{3,\max}$
            \EndIf
            \State $L_{\min}\gets
            L_{\mathrm{fixed}}+\epsilon a_{3,\mathrm{start}}$
            \State $pD\gets p\cdot D$
            \Comment{precomputed baby step}
            \State BabyStepsToIndices $\gets$ HashTable()

            \For{$i\in\{0,\ldots,r\}$}
                \State Compute $ip\cdot D$ by adding the precomputed $pD$ at each step.
                \State BabyStepsToIndices[Hash($ip\cdot D$)]
                $\gets$
                BabyStepsToIndices[Hash($ip\cdot D$)] $\cup \{ip\}$
            \EndFor

            \State $rpD\gets rp\cdot D$ and compute $L_{\min}\cdot D$

            \For{$j\in\{0,\ldots,s\}$}
                \State Giant $\gets (L_{\min}+jrp)\cdot D$,
                computed by adding the precomputed $rpD$ at each step.
                \If{Hash(Giant) in BabyStepsToIndices}
                    \For{$ip\in$ BabyStepsToIndices[Hash(Giant)]}
                        \State CandidateOrder $\gets L_{\min}+jrp-ip$
                        \State $a_3\gets
                        \epsilon(\text{CandidateOrder}-L_{\mathrm{fixed}})$
                        \If{$(a_1,a_2,a_3)\in$ CandidateTriples}
                            \State CandidateTriplesTmp
                            $\gets$
                            CandidateTriplesTmp
                            $\cup\{(a_1,a_2,a_3)\}$
                        \EndIf
                    \EndFor
                \EndIf
            \EndFor
        \EndFor
        \State \Return CandidateTriplesTmp

        \EndFunction
    \end{algorithmic}
\end{algorithm}

\subsection{Correctness of the algorithm}

We first show 
in Lemma~\ref{lem:BSGS-correct} that the correct candidate is among the candidates returned by Algorithm~\ref{alg:step-p-BSGS}.
We then show that this list contains at most $\loopsize$ candidates in Lemma~\ref{lem:step2-loopsize}.

\begin{lemma}\label{lem:BSGS-correct}
    The output of Algorithm~\ref{alg:step-p-BSGS} contains the correct group order.
\end{lemma}
\begin{proof} We consider Algorithm~\ref{alg:step-p-BSGS}, which performs one round of baby-step giant-step. 
For a fixed choice of $a_2$ and $\epsilon\in\set{1,-1}$, let $L_{\text{fixed}} = (p^3+1)+\epsilon a_1(p^2+1)+a_2(p+1)$. 
Here $\epsilon$ corresponds to the choice of $\Jac(C)(\F_p)$ or the trace zero variety.
When $\epsilon=1$, let $a_{3,\mathrm{start}}=a_{3,\min}$, and when
$\epsilon=-1$, let $a_{3,\mathrm{start}}=a_{3,\max}$.
Set
\[L_{\min}=L_{\text{fixed}}+\epsilon a_{3,\mathrm{start}}.\] The possible group orders determined 
by different $a_3$ values
form an arithmetic sequence with common difference $p$ starting at $L_{\min}$. Using baby-step giant-step,
we search for the index $k$ for which $L_{\min} + kp$ annihilates a random divisor $D$.

We set $r=s=\lceil\sqrt{40}p^{1/4}\rceil$. The baby steps are $ip$ where $0\leq i \leq r$. The giant steps are $(L_{\min}+jrp)$ where $0\leq j \leq s$.
If we find a collision $(L_{\min} +jrp) \cdot D=ip \cdot D$,
then $(L_{\min}+jrp-ip)\cdot D =0$ and $L_{\min}+jrp-ip$ is a contender for the group size.
The corresponding $a_3$ is given by $\epsilon\bigl(L_{\min}+jrp-ip-L_{\text{fixed}}\bigr)$. By the choice of $r$ and $s$, the 
search space contains the correct Jacobian group order. Hence,
we are guaranteed to find the collision corresponding to the correct group order.
\end{proof}

\begin{lemma}\label{remark:step2-survivors}
    For each candidate triple $(a_1,a_{2,i},a_{3,i})$, write $L_i(T) := L_{p,a_1,a_{2,i},a_{3,i}}(T)$.
    Let $\lambda$ and $\lambda'$ be the exponents of the groups $\Jac(C)(\F_p)$ and $\Jac(C)(\F_{p^2})/\Jac(C)(\F_p)$, respectively. 
    The candidates 
    that survive arbitrarily many rounds of Algorithm~\ref{alg:BSGS}
    are precisely the ones for which $\lambda\mid L_i(1)$ and $\lambda'\mid L_i(-1)$.

\end{lemma}
\begin{proof}
    Let $L_i(T)$ be a candidate for the $L$-polynomial, with the property that $\lambda\nmid L_i(1)$. Then, there 
    must exist elements of $\Jac(C)(\F_p)$ that are not annihilated by $L_i(1)$. Otherwise the group exponent 
    would be a divisor of $\gcd(L_i(1), \lambda) < \lambda$, and we obtain a contradiction. Assuming uniform randomness, 
    we would eventually obtain a divisor $D$ that is not annihilated by $L_i(1)$, and $L_i(1)$ would be eliminated.
    On the other hand, if $\lambda\mid L_i(1)$, then $L_i(1)$ annihilates all group elements of $\Jac(C)(\F_p)$.

    The above statement also applies to the trace zero variety. This proves the claim.
\end{proof}

\begin{lemma}\label{lem:step2-loopsize}
    Let $(a_1,a_{2,i}, a_{3,i})_{i\in\set{1,\ldots,k}}$ be the triples from Step 2's input. For each $i$, 
    write $L_i(T) := L_{p,a_1, a_{2,i}, a_{3,i}}(T)$. Let $\lambda$ be the group exponent of $\Jac(C)(\F_p)$
    and $\lambda'$ be the group exponent of $\Jac(C)(\F_{p^2})/\allowbreak\Jac(C)(\F_p)$. Assume $p>1600$.
    Then there are at most $\loopsize$ candidates $i$ for which $\lambda\mid L_i(1)$ and $\lambda'\mid L_i(-1)$.
\end{lemma}

\begin{proof}
    Recall that the values $L_i(1)$ must lie in one of the following $6$ intervals (we assume $p>1600$ so that these intervals are disjoint):
        \[[B+ kp(p+1)-20p^{3/2}, B+ kp(p+1)+20p^{3/2}],\]
        where $0\leq k\leq 5$ and $B=(p^3+1)+a_1(p^2+1)+ a_{2,\min}(p+1)$, and $a_{2,\min}$
        is the smallest candidate for $a_2$ from Step 1.
        If all intervals contain at most one element, then 
        there are at most $6$ candidates. Otherwise, two contenders lie in the same interval. Let them be $(a_1,a_{2,i},a_{3,i})$
        and $(a_1,a_{2,j},a_{3,j})=(a_1, a_{2,i}+\delta_2p, a_{3,i}+\delta_3p)$, yielding $L_i(T)$ and $L_j(T)$, respectively.
         By Lemma~\ref{remark:step2-survivors}, we have $\lambda\mid L_i(1)$ and $\lambda\mid L_j(1)$.
         Hence $\lambda\mid L_j(1)-L_i(1)=p(\delta_2(p+1)+\delta_3)$, but since
        they are in the same interval, $\delta_2=0$, and hence $\lambda\mid p\cdot \delta_3$.

        Denote the image of the map $\iota:\Jac(C)(\F_p)\hookrightarrow \Jac(C)(\F_{p^2})$ by $J$. We
        identify the trace zero variety with the subgroup given by the image of $\phi: \Jac(C)(\F_{p^2})/\Jac(C)(\F_{p})\to \Jac(C)(\F_{p^2}), x\mapsto F(x)-x$, 
        and call it $J'$. This identification is justified by the isomorphism $\Jac(C)(\F_{p^2})/\Jac(C)(\F_p)\simeq\operatorname{Im}(F-1)$ in Section~\ref{section-lifting}. Furthermore, $J'\subset \ker(F+1)$, 
        because if $x\in J'$, then $x=F(y)-y$ for some $y\in\Jac(C)(\F_{p^2})$.
        Then $F(x)+x = F(F(y)-y) + F(y)-y = 0$.

        From the above results, we see that  $J\cap J' \subset \ker(F-1)\cap \ker(F+1)\subset \Jac(C)(\F_{p^2})[2]$.

        Now, let $H$ and $H'$ be the $p$-coprime subgroups of $J$ and $J'$, respectively. We next obtain some bounds on the order of $H\times H'$.

        By the Hasse--Weil bound for $\Jac(C)(\F_{p^2})$, we have $\#\Jac(C)(\F_{p^2}) \geq (p-1)^6$. The $p$-rank of $\Jac(C)(\F_{p^2})$
        is at most 3 (see page~147 in \cite{Mum08} or Tag~03RP in \cite{Sta25} for a proof). Since $\lambda\mid p\cdot\delta_3$ and $\lambda'\mid p\cdot\delta_3$, we have $\lambda\leq p\abs{\delta_3}< p\cdot 40\sqrt{p}<p^2$, 
        and likewise for $\lambda'$. This implies that $\Z/p^2\Z$ is not a subgroup of $\Jac(C)(\F_{p^2})$.
        We also see that the maximum $r$ such that $p^r\mid \abs{J}\abs{J'}$ is at most $3$.

        Therefore, $\#\Jac(C)(\F_{p^2})\leq \abs{H}\abs{H'}\cdot p^3$.

        On the other hand, consider the map $\tau: H\times H'\to \Jac(C)(\F_{p^2})[\delta_3]$, via $(x,y)\mapsto x-y$. Then 
        \[\abs{H}\abs{H'}=\abs{H\times H'} =\abs{\ker\tau}\abs{\operatorname{Im}\tau} \leq 2^6\abs{\delta_3}^6.\]
        Combining this with the earlier inequalities, we obtain that $\abs{\delta_3}\geq \frac{p-1}{2\sqrt{p}}$. Given that the interval in which $\delta_3$ can lie 
        has length at most $40\sqrt{p}$,  
        the total number of candidates in this interval is at most $81$. There are at most $6$ intervals, hence at most $\loopsize$ candidates.

\end{proof}

\subsection{Complexity analysis}

The space complexity for the Step 2 algorithm is $O(p^{1/4}\log p)$.

Although the search space (the possible triples consistent with $L_p(T)\bmod p$)
 has size $O(\sqrt{p})$, the baby-step giant-step search only stores $O(\sqrt{\sqrt{p}})=O(p^{1/4})$ 
 baby steps in the hash table, each of size $O(\log p)$ bits, for $O(p^{1/4}\log p)$ space.

 The following proposition measures the time complexity in terms of the number of group operations.
\begin{proposition}
    Each call to Algorithm~\ref{alg:step-p-BSGS} uses 
    $2\lceil\sqrt{40}\,p^{1/4}\rceil + O(\log p) = O(p^{1/4})$ group operations. 
\end{proposition}
\begin{proof}
    To compute $p\cdot D$, $rp\cdot D$, and $L_{\min}\cdot D$, we only need $O(\log p)$ group operations.
    The dominating steps are the two loops that accumulate the baby steps and giant steps.
    The first loop takes $r$ group operations and the second takes $s$, with $r,s\sim\sqrt{40}p^{1/4}$.
\end{proof}

\section{Step 3: Sylow-subgroup elimination}\label{section-step3}

When Step 2 does not yield a unique candidate (which happens rarely, see Remark~\ref{experiment-step3examples}), we proceed to Step 3
and Step 4. We present two probabilistic algorithms (of Las Vegas type) in this section and the next. In each of these sections, we prove their correctness and discuss their 
complexity. 

The main idea behind Step 3 (Sylow-subgroup-eliminator) is the following:  it repeatedly compares different pairs of candidates $(a_1,a_{2,i}, a_{3,i})$ and  $(a_1,a_{2,j}, a_{3,j})$ (yielding $L_i(T)$ and $L_j(T)$,
respectively). If there is a small prime $\ell$ and an integer $b\geq 0$ such that $\ell^{b+1}\mid L_i(1)$ but $\ell^{b+1}\nmid L_j(1)$, then it
attempts to prove the existence of a size $\ell^{b+1}$ subgroup and eliminates $L_j$ if successful.

Recall that the time complexity for the first two steps combined is $O(p^{1/4+o(1)})$, and that most 
cases terminate at Step 2. Hence, we would like to have an algorithm that finishes within  $O(p^{1/4+o(1)})$ time,
which makes the runtime of the whole algorithm $O(p^{1/4+o(1)})$.
The Step 3 Sylow-subgroup-eliminator method works very fast in practice, and one could use the Sylow-subgroup-eliminator
to eliminate all but one candidate. However, the most expensive ingredient is to prove the existence of a size $\ell^{b+1}$
subgroup, and the best complexity we could  prove for this step
is $O(p^{1/2+o(1)})$, which exceeds $O(p^{1/4+o(1)})$.
Therefore, to maintain our overall time complexity of $O(p^{1/4+o(1)})$, we
use a budgeted version of the Sylow-subgroup-eliminator (so it terminates in $O(p^{1/4+o(1)})$),
and prove termination under a weaker condition --- instead of having one candidate left, we have 
at most $6$ candidates left.

The main idea behind Step 4 is the following: once we exit Step 3 with 
at most $6$ candidates --- more precisely, at most one candidate per $a_2$ value --- we
compute $L_p(T)\bmod 2$ and/or $L_p(T)\bmod 3$ and find the unique correct candidate in $O(p^{o(1)})$ time.
Therefore, the overall 
time complexity is within $O(p^{1/4+o(1)})$.

\begin{remark}
In our experiments, all the cases terminate no later than Step 3. See Subsection~\ref{Subsection-step3examples}.
    
\end{remark}

\subsection{The algorithm}

Algorithm~\ref{alg:sylow-eliminator} takes the remaining candidates from Step 2
 and proceeds iteratively.
For each iteration, it compares two different candidates $(a_1,a_{2,i}, a_{3,i})$ and  $(a_1,a_{2,j}, a_{3,j})$ (yielding $L_i(T)$ and $L_j(T)$, 
respectively). 
By calling the Budgeted-Sylow-subgroup-verifier subroutine described in Remark~\ref{rem:sylow-verifier}, it 
attempts to show that $L_i(T)$ is correct by finding a subgroup of order $\ell^{b+1}$, where $\ell^{b+1}\mid L_i(1)$ and $\ell^b\parallel L_j(1)$.
If such a subgroup exists, 
it eliminates all candidate group orders not divisible by $\ell^{b+1}$, including $L_j(T)$. It repeats 
the comparison between every pair of candidate group orders until the $a_2$ values of the surviving candidates are pairwise distinct.

\begin{remark} \label{rem:sylow-verifier}

    For a fixed prime $\ell$ and an integer $b\geq 0$, the \textbf{Budgeted-Sylow-subgroup-verifier} subroutine takes a set of group elements 
    $S=\set{\beta_1,\ldots,\beta_8}$ 
    (each having order a power of $\ell$)
    as input, and attempts to verify whether $\langle S \rangle$ contains a subgroup of size at least $\ell^{b+1}$.
    This algorithm is obtained by modifying Algorithm~4 in \cite{Sut10} as follows. 
    
    For Step 2 of that algorithm, as soon as 
    $h_i=\log_\ell \abs{\beta_i}$ satisfies $h_i\geq b+1$, we return \textbf{true}.
    Likewise, as soon as the basis $\boldsymbol{\alpha}$ satisfies
    $\prod_{\alpha_j\in\boldsymbol{\alpha}} \abs{\alpha_j}\geq \ell^{b+1}$, 
    we return \textbf{true}. Otherwise, when every $h_i$ is $0$ but the product 
    of the orders of the elements in $\boldsymbol{\alpha}$ is not at least $\ell^{b+1}$, we return \textbf{false}.

    Before we execute Step 3 of that algorithm, let $s=\#\boldsymbol{\alpha}$. If $\ell^s>40\sqrt{p}$, we return \textbf{abort}.
    Indeed, when Step 3 is executed with a basis of size $s$, the 
    discrete logarithm step costs $O(\ell^{s/2})$. 
    Hence every non-aborted call to Step 3 has at most $\ell^{s/2}\leq \sqrt{40}p^{1/4}$ group operations,
    which does not exceed our budget.
    Therefore, \textbf{Budgeted-Sylow-subgroup-verifier} takes $O(p^{1/4+o(1)})$ group operations. 
\end{remark}

\begin{algorithm}%
    \caption{Sylow-subgroup-eliminator}\label{alg:sylow-eliminator}
    \begin{algorithmic}   
    \Function{Sylow-subgroup-eliminator}{}
    \State \textbf{Input:} a list of candidate $L$-polynomials $L_{i}=L_{p, a_1, a_{2,i},a_{3,i}}(T)$, $1\leq i\leq k$,
    $k\leq \loopsize$, and the groups $\Jac(C)(\F_p)$ and $\Jac(C)(\F_{p^2})$.
    \State \textbf{Output:} a reduced list of candidates, where there is at most one candidate for each value of $a_2$
    \State $G_1 \gets \Jac(C)(\F_p)$
    \State $G_{-1} \gets \operatorname{Im}(F-1)\subseteq \Jac(C)(\F_{p^2})$
    \State $E\gets \emptyset$ \Comment{Indices of eliminated candidates}
    \Repeat
    \For{$\epsilon \in\set{1,-1}$}
    \For{$i\in \{1,\ldots,k\}$} 
    \For{$j\in \{1,\ldots,k\}$}
     \If{$i\in E$ or $j\in E$ or $j=i$}
        \State \textbf{continue}
    \EndIf
    \If{$\#E=k-1$}
        \State \Return{the unique candidate $L_i$ with $i\notin E$}
    \EndIf
    \State \Comment{Compare $L_i(\epsilon)$ and $L_j(\epsilon)$}
    \For{$(\ell, b)\in \set{(\ell, b): \ell^{b+1}\mid L_i(\epsilon), \ell^b \parallel L_j(\epsilon)}$}
    \State \Comment{Attempt to show the existence of a subgroup of size $\ell^{b+1}$.}
    \State $S\gets\set{\beta_1,\ldots,\beta_{8}}$, where the $\beta_t$ are random elements of the $\ell$-Sylow subgroup of $G_\epsilon$ 
    \State $\mathit{Res}\gets$ \Call{Budgeted-Sylow-subgroup-verifier}{$G_\epsilon, \ell^{b+1}, S$}
    \State \Comment{Returns \textbf{true}, \textbf{false}, or \textbf{abort}. Only \textbf{true} is acted on}
    \If{$\mathit{Res}=\textbf{true}$}
            \State $E \gets E\cup\set{m : \ell^{b+1}\nmid L_m(\epsilon)}$
            \Comment{This eliminates $j$ and possibly more indices.}
            \State \textbf{break}
    \EndIf
    \EndFor
    \EndFor
    \EndFor
    \EndFor
    \Until{the values $a_{2,i}$, for $i\notin E$, are pairwise distinct}
    \State \Return $\{L_i: 1\leq i\leq k, i\notin E\}$
    \EndFunction
    \end{algorithmic}
\end{algorithm}

\subsection{Proof of correctness}

We first show that Algorithm~\ref{alg:sylow-eliminator} terminates.

\begin{lemma}\label{lemma:no-div}
    Among all candidate triples from Step 1, when $p>1600$, there do not exist $i\neq j$ such that $L_i(\epsilon)\mid L_j(\epsilon)$, for 
    $\epsilon\in\set{\pm 1}$.
\end{lemma}
\begin{proof}
    Suppose, towards a contradiction, that there exists a pair of indices such that $L_i(\epsilon)\mid L_j(\epsilon)$.
    Let these triples be $(a_1, a_2, a_3)$ and $(a_1, a_2', a_3')$, respectively, with $a_2'=a_2+p\delta_2, a_3'=a_3+p\delta_3$.
    
    Using results from Step 1, we see that for $p>1600$, \[\abs{L_i(\epsilon)-L_j(\epsilon)}\leq p\abs{\delta_2(p+1)}+ \abs{p\delta_3} \leq 5p(p+1)+40p^{3/2}\leq 10p^2 + p^2 \leq 11p^2< \frac12 p^3.\]

    On the other hand, for $p>1600$,
    \[L_i(\epsilon) = p^3 + 1 + \epsilon a_1(p^2+1)+a_2(p+1)+\epsilon a_3 \geq p^3 - 6\sqrt{p}(p^2+1)-15p(p+1)-20p^{3/2}>\frac12 p^3.\]
    Therefore, if $L_j(\epsilon)$ is a multiple of $L_i(\epsilon)$, then $L_j(\epsilon)\geq 2L_i(\epsilon)$, 
    and we obtain the contradiction \[\frac12 p^3>\abs{L_j(\epsilon)-L_i(\epsilon)}\geq L_i(\epsilon)>\frac 12 p^3.\]
\end{proof}

\begin{proposition}\label{prop:sylow-terminates}
    Assume $p>1600$. Then Algorithm~\ref{alg:sylow-eliminator} must terminate with
    a reduced list of candidate triples whose $a_2$ values are all distinct.
\end{proposition}

\begin{proof}

    Suppose, towards a contradiction, that there exist two candidate triples with the same $a_2$ value,
    $(a_1, a_2, a_3)$ and $(a_1, a_2, a_3')=(a_1,a_2,a_3+\delta_3 p)$, that are not eliminated in the process. We denote their 
    respective $L$-polynomials  
    by $L_p(T)$ and $L_p'(T)$. We denote the true $L$-polynomial by $L(T)$.

    We first claim that there exists a prime $\ell$ such that 
    \[k:= v_\ell(L(1)) > \min\set{v_{\ell}(L_p(1)), v_{\ell}(L_p'(1))}.\]

    Suppose not. Then $L(1)$ would divide both $L_p(1)$ and $L_p'(1)$, contradicting Lemma~\ref{lemma:no-div}.

    Set $a:= \max\set{v_{\ell}(L_p(1)), v_{\ell}(L_p'(1))}$ and $b:= \min\set{v_{\ell}(L_p(1)), v_{\ell}(L_p'(1))}$.

    If $a>b$, then when the algorithm compares $L_p(1)$ and $L_p'(1)$, it attempts to certify 
    a subgroup of size $\ell^{b+1}$. Since $k>b$, such a subgroup exists in the true group.

    If $a=b$, then the algorithm will not attempt to show a subgroup of size $\ell^{b+1}$
    when comparing these two. However, when comparing one of them with the true candidate, 
    it will invoke the test for a subgroup of size $\ell^{b+1}$.
    
    Hence, we need to show that the tests successfully eliminate one of $L_p, L_p'$
    within the prescribed budget.

    We consider two cases:

    \underline{Case 1: $\ell\neq p$}

    If $b=0$, it suffices to find one nontrivial element of the $\ell$-Sylow subgroup, which can be done 
    in logarithmic time. So we assume $b\geq 1$.

    Since neither $L_p(T)$ nor $L_p'(T)$ was eliminated in Step 2, we may assume by Lemma~\ref{remark:step2-survivors} that every group element in $\Jac(C)(\F_p)$
    is annihilated by $L_p(1)$ and $L_p'(1)$. Hence the group exponent $\lambda$ divides $\abs{L_p(1)-L_p'(1)}=\abs{\delta_3 p}$.
    
    Let the group exponent of the $\ell$-Sylow subgroup of $\Jac(C)(\F_p)$ be $\ell^t$. We have $\ell^t\mid \ell^{b+1}\mid \abs{\delta_3 p}$.
    But since $\ell\neq p$, we have $\ell^{b+1}\mid \abs{\delta_3}$, where $\abs{\delta_3}\leq 40\sqrt{p}$.
    
    The subroutine Budgeted-Sylow-subgroup-verifier
    allows enough budget to show the existence of such a subgroup. Therefore, Algorithm~\ref{alg:sylow-eliminator}
    would have eliminated at least one of the candidates.

    \underline{Case 2: $\ell=p$}

    If $b=0$, then it suffices to find a nontrivial element of the $p$-Sylow subgroup, which can be done in logarithmic time.
    We assume $b\geq 1$. In this case, the exponent of the $p$-Sylow subgroup divides $\abs{\delta_3 p}$. Since $\abs{\delta_3}<40\sqrt{p}$,
    the exponent must be exactly $p$.

    The above means that we need to exhibit a subgroup of size $p^2$ or $p^3$. Since 
    the group exponent is $p$, this means that we must exhibit a subgroup isomorphic to $\Z/p\Z\times\Z/p\Z$.
    In this case, we would have turned to the trace zero variety, where this phenomenon cannot occur. This is because
    if the trace zero variety also has a subgroup isomorphic to  $\Z/p\Z\times\Z/p\Z$, then $\Jac(C)(\F_{p^2})$ would 
    have $p$-rank at least 4. However, it is well known that the $p$-rank cannot exceed the genus (see page~147 in \cite{Mum08} or Tag~03RP in \cite{Sta25} for a proof).
     This concludes the proof.

\end{proof}

Having established termination, we next show that the correct candidate is never eliminated.

Without loss of generality, we assume that $L_1(1)$ is the correct group order. 
From the correctness of Step 2 (Algorithm~\ref{alg:BSGS}), $L_1(1)$ 
must be in the input list of candidate group orders. From the correctness of the
 Budgeted-Sylow-subgroup-verifier subroutine (see Remark~\ref{rem:sylow-verifier}), we only eliminate $L_j(1)$ 
if there exists a subgroup of size $\ell^{b+1}$ but $\ell^{b+1}\nmid L_j(1)$. Therefore, 
$L_1(1)$ will never be eliminated.

\subsection{Complexity analysis}

By Lemma~\ref{lem:step2-loopsize}, there are $O(1)$ many ordered pairs of candidates $(i,j)$ being compared (attempting to prove that $L_i$ is the correct candidate).  
By Remark~\ref{rem:sylow-verifier}, the subroutine Budgeted-Sylow-subgroup-verifier takes $O(p^{1/4+o(1)})$ group operations. 
In the case that $i$ is the correct candidate, a subgroup of the required size $\ell^{b+1}$ exists, and by Theorem~3.1 in \cite{Pak00},
since the rank is at most $6$, the $8=6+2$ random elements of the $\ell$-Sylow subgroup of $G_\epsilon$ generate such a subgroup with probability at least $\frac{1}{2}$.
Hence the expected number of invocations of the subroutine Budgeted-Sylow-subgroup-verifier is at most $2$.

Therefore, the overall expected number of group operations is $O(p^{1/4+o(1)})$.

\subsection{Step 3 examples} \label{Subsection-step3examples}

\begin{remark}\label{experiment-step3examples}
    With an implementation of Steps 1 and 2 in Magma \cite{MAGMA} and access to Sutherland's database \cite{Sut19}, we investigated how often Step 3 is needed. 
    For primes between $1600$ and $2048$, among approximately 3,712,000 random pairs of a curve and a prime of good reduction, 
    only $180$ required Step 3 (roughly one in every 20,622 instances).
    For primes between $2048$ and $4096$, among approximately 3,264,000 random pairs of a curve and a prime of good reduction, 
    only $291$ required Step 3 (roughly one in every 11,216 instances).
\end{remark}
From the experiments in Remark~\ref{experiment-step3examples} and a few manual searches, we collected a list of $422$ distinct examples where Step 3 was needed, which all follow a particular form:
\begin{itemize}
    \item One candidate is always of the form $L_E(T)^3$,
where $L_E(T)$ is a degree $2$ polynomial of the form $c_2T^2 - c_1T + 1$ with $c_1,c_2\in\Z$. 
This suggests that $L_E(T)$ is the $L$-polynomial of an elliptic curve $E$ over $\F_p$ 
and $\Jac(C)$ is isogenous to $E\times E \times E$.
    \item All the other candidates are of the form $L_E(T)P(T)$,
where $P(T)$ is a degree $4$ polynomial with $P(T)\neq L_E(T)^2$.
\end{itemize}
For example, when $p=1697$ and $f =  x^4 + 3x^2y^2 + 2y^4 + 2x^3z + 3x^2yz + 3xy^2z + 4y^3z + 4x^2z^2 + 3xyz^2 + 7y^2z^2 + 3xz^3 + 5yz^3 + 3z^4$,
we have $a_1=126$ and the list of candidates is
\[\begin{cases}
    (a_2,a_3)=(6989, 359184)&\Rightarrow L_{p}(T) = (pT^2 + 42T + 1)(p^2T^4 + 142548T^3 + 1764T^2 + 84T + 1)\\
    (a_2,a_3)= (8686, 430458)&\Rightarrow L_{p}(T) = (pT^2 + 42T + 1)(p^2T^4 + 142548T^3 + 3461T^2 + 84T + 1)\\
    (a_2,a_3)=(10383, 501732 )&\Rightarrow L_p(T)= (pT^2 + 42T + 1)^3
\end{cases}.\]

All the candidates for the $L$-polynomials have a common factor $L_E(T) := (pT^2 + 42T + 1)$,
where $L_E(1) = 1740$. It is very likely that the group exponent of $\Jac(C)(\F_p)$ is a factor of $1740$. In this case, 
$1740$ annihilates all random points on $\Jac(C)(\F_p)$.
Hence Step 2 could not distinguish any of these three candidates.

Now, if the correct $L$-polynomial is the first or the second candidate, then we would very likely have generated 
a point whose order is not a factor of $1740$, and would thereby have eliminated the other two.
 Therefore, the probability of the first or the second candidate being correct is extremely small.

A very natural strategy is to work with the trace zero variety $\Jac(C)(\F_{p^k})/\Jac(C)(\F_p)$ for $k=2,3,4,\ldots$:
the group order of the trace zero variety for the candidate $L_i(T)$ is $T_{i,k} := \prod_{\zeta^k=1, \zeta\neq 1} L_i(\zeta)$.
We could generate a random point $D\in \Jac(C)(\F_{p^k})$ and eliminate all $L_i$ such that $T_{i,k}\cdot D \notin \Jac(C)(\F_p)$.
However, in the most likely case that the third candidate is correct, the above strategy will not eliminate any further candidates.
The group order of the trace zero variety is $\left(\prod_{\zeta^k=1,\zeta\neq 1} L_{E}(\zeta)\right)^3$, and 
the group exponent will be a factor of $\left(\prod_{\zeta^k=1,\zeta\neq 1} L_{E}(\zeta)\right)$, which is also
a factor of all the other candidate group orders of the trace zero variety.

Hence, Step 3 is needed to find the correct candidate with proof, though the probability 
of the first two candidates being correct is extremely small. 
We emphasize that in our experiments, all $422$ examples requiring Step 3 share this pattern.

\section{Step 4: $L_p(T)$ modulo small primes} \label{section-lv}

When Step 1, Step 2, and Step 3 do not yield a unique candidate, we use Step 4. In this step, we 
use a Las Vegas algorithm to compute $L_p(T)\bmod \ell$, where $\ell$ is a small prime. In Proposition~\ref{prop:sylow-terminates},
we have proven that all the outcomes from Step 3 have different $a_2$ values, and by Proposition~\ref{prop:ell-2-3-suffice}, we see that 
computing $L_p(T)\bmod \ell$ for $\ell=2$ and $\ell=3$ is enough to single out the correct candidate. In other words, Step 4
is guaranteed to produce a single winner. In Subsection~\ref{Subsection-LV-complexity},
we see that this algorithm has $O(p^{o(1)})$ runtime. Hence the overall algorithm still has $O(p^{1/4+o(1)})$ runtime.

\subsection{The algorithm}\label{Subsection-LV-algo}

The main idea of this step is to use different small primes $\ell$ and compute $L_p(T)\bmod \ell$, then 
eliminate the candidates whose mod $\ell$ reduction is incorrect. In all the cases where Step 3 is needed, using $\ell=2$ and $\ell=3$ is sufficient. 
 Moreover, in the cases where all the surviving candidates have different $a_2$ values, $\ell\in\set{2,3}$ 
suffices. (This is because, if $a_2$ and $a_2'$ are two different values, then $a_2'=a_2+\delta_2p$, where $\delta_2$ is a nonzero integer with absolute value at most $5$.)

Once we select $\ell$, we compute $L_p(T)\bmod \ell$ using Algorithm~\ref{alg:L-poly-mod-ell}. In this algorithm, we first compute the safe extension degree $K_i$ for each candidate $L_i(T)$, which would be a 
multiple of the torsion field extension degree if $L_i(T)$ were the correct candidate. See Definition~\ref{defn:k-degree} and Proposition~\ref{prop:safe-degree} for details. We then iterate through 
possible degrees of the torsion field extension $k$, and attempt to enumerate the full $\ell$-torsion over the field, i.e., $\Jac(C)(\F_{p^k})[\ell]$. This step could
either succeed or fail. We also prune the extension degrees $k$ where $\ell^6\nmid \#\Jac(C)(\F_{p^k})=\abs{\operatorname{Res}(T^k-1,L_i(T))}$, as 
then the group cannot contain the full $\ell$-torsion.
\begin{itemize}
    \item If this step is successful,
we also get a basis $e_1,\ldots, e_6\in \Jac(C)(\F_{p^k})[\ell]$ and a hash table that maps each of the $\ell^6$ elements to its 
corresponding vector. Using this information, we can compute the characteristic polynomial of the Frobenius matrix,
 corresponding to $F(e_1),\ldots, F(e_6)$. This is the subroutine \textit{Compute-$L$-poly-mod-$\ell$-helper} in Algorithm~\ref{alg:L-poly-mod-ell}.
    \item If this step fails, we continue by trying the next possible extension degree. 

\end{itemize}

\begin{algorithm}
    \caption{Compute $L_p(T)\bmod \ell$}\label{alg:L-poly-mod-ell}
    \begin{algorithmic}   
        \Function{Compute-$L$-poly-mod-$\ell$}{$\{L_i\}_{i\in I},\Jac(C),\ell$}
        \State \textbf{Input:} a list of candidate $L$-polynomials $L_i=L_{p, a_1, a_{2,i}, a_{3,i}}(T), i\in I$, $\Jac(C)$, $\ell$
        \State \textbf{Output:} $L_p(T)\bmod \ell$
        \For{$i\in I$}
        \State $K_i\gets$ \Call{Safe-Extension-Degree}{$L_i(T),\ell$}
        \EndFor
        \State $\textit{Degree-tries}\gets \operatorname{Sort}\left(\bigcup_{i\in I}\{k:k\mid K_i\}\right)$
        \Repeat
        \For{$k\in\textit{Degree-tries}$}
        \If{$\forall i\in I$, $\ell^6\nmid|\operatorname{Res}(T^k-1,L_i(T))|$}
        \State\textbf{continue}
        \EndIf
        \State (\textit{success}, \textit{hash-to-vec}, \textit{basis}) $\gets$ \Call{Attempt-to-build-$\ell$-torsion}{$k,\ell,\Jac(C),\set{L_i}_{i\in I}$}
        \If{\textit{success}}
        \State $L_p(T)\bmod\ell \gets$ \Call{Compute-$L$-poly-mod-$\ell$-helper}{\textit{basis}, \textit{hash-to-vec}, $\Jac(C)$}
        \State \Return  $L_p(T)\bmod\ell$
        \EndIf
        \EndFor
        \Until{\textit{success}}
    \EndFunction
    \end{algorithmic}
\end{algorithm}

We now define the safe extension degree used in Algorithm~\ref{alg:L-poly-mod-ell} and prove that it has the required property.

\begin{definition}\label{defn:k-degree}
    Let $L_p(T)$ be the correct $L$-polynomial for $\Jac(C)$. Let $\ell\neq p$ be a prime number. 
    In our work, we only use $\ell\in\set{2,3}$.
    Let $\chi_F(X) = X^6L_p(1/X)\bmod \ell \in \F_\ell[X]$ and write \[\chi_F(X)=\prod_{j=1}^h f_j(X)^{e_j}\]
    where each $f_j(X)\in \F_{\ell}[X]$ is an irreducible polynomial of degree $d_j$. Let $r=\ceil{\max_j \log_{\ell} e_j}$.

    Then, the safe extension degree $K$ for $\ell$ is defined as
    \[K = \ell^rM, \text{ where }M=\prod_{j=1}^{h} (\ell^{d_j}-1).\]
\end{definition}

\begin{remark}
    The integer $K$ is called the safe extension degree because $\Jac(C)(\F_{p^{K}})$ is guaranteed to have full rational $\ell$-torsion, but 
    often the minimal such extension degree is much smaller than $K$. It is always a divisor of $K$. Therefore, in Algorithm~\ref{alg:L-poly-mod-ell}, 
    we enumerate over the divisors of $K$ to obtain the smallest field extension possible.
\end{remark}

\begin{proposition}\label{prop:safe-degree}
    Let $K$ be defined as in Definition~\ref{defn:k-degree}. Then, $\Jac(C)[\ell]\subseteq \Jac(C)(\F_{p^K})$.
\end{proposition} 

\begin{proof}
    Let $F\in \F_\ell^{d\times d}$ be the matrix of the Frobenius action on $\Jac(C)[\ell]$. In our case, $d=2g=6$. To show
    that $\Jac(C)[\ell]\subseteq \Jac(C)(\F_{p^K})$, it suffices to show $F^K(P)=P$ for all $P\in \Jac(C)[\ell]$. 

    We first show that $\forall j, f_j(X)\mid X^M-1$. Let $\alpha$ be a root of $f_j(X)$. Since $\deg(f_j)=d_j$ and $\ell\neq p$, $\alpha\in\F_{\ell^{d_j}}^\times$.
    Equivalently, $\alpha$ is a root of $X^{\ell^{d_j}-1}-1$. Since $f_j$ is irreducible, it is the minimal polynomial of $\alpha$ over $\F_\ell$. Hence $f_j(X)\mid X^{\ell^{d_j}-1}-1$. Since $\ell^{d_j}-1\mid M$, $f_j(X)\mid X^M-1$.

    We next observe that in characteristic $\ell$, we have $(X^M-1)^{\ell^r} = X^{M\ell^r}-1=X^K-1.$

    Finally, we see that \[f_j(X)^{e_j}\mid f_j(X)^{\ell^r} \mid (X^M-1)^{\ell^r} = X^K-1,\]
    and since the $f_j(X)$ are distinct irreducibles (hence pairwise coprime), we conclude that \[\chi_F(X)=\prod_{j=1}^{h} f_j(X)^{e_j} \mid X^K-1.\]
    By the Cayley--Hamilton theorem, $\chi_F(F)=0$, so $F^K=I$. This finishes our proof.
\end{proof}

\begin{proposition}\label{prop:ell-2-3-suffice}
    Suppose we use Algorithm~\ref{alg:L-poly-mod-ell} to compute $L_p(T)\bmod 2$ and $L_p(T)\bmod 3$ and 
    eliminate the incompatible $L$-polynomial candidates from Step 3.
    Then, we can determine the unique candidate.
\end{proposition}
\begin{proof}
    Suppose there exist two distinct candidate triples $(a_1, a_2, a_3)$ and $(a_1,a_2', a_3')$, corresponding to 
    $L_p(T)$ and $L_p'(T)$, respectively. We know that $a_2\neq a_2'$.
    Let $a_2'=a_2+\delta_2 p$, where $\abs{\delta_2}<6$. If $L_p(T)\equiv L_p'(T)\bmod 2$, then $2\mid \delta_2$.
    If $L_p(T)\equiv L_p'(T)\bmod 3$, then $3\mid \delta_2$. This is a contradiction since $\abs{\delta_2}<6$.
\end{proof}

\subsection{Proof of correctness}

The correctness of the computation of $L_p(T)\bmod \ell$ follows from the correctness of Sutherland's algorithm \cite{Sut10}
that computes the $6$ basis elements of the $6$-dimensional vector space $\Jac(C)(\overline{\F}_{p})[\ell]$
with $\F_\ell$ coefficients. Then, discrete logarithms express each Frobenius image as a linear combination of 
the $6$ basis elements, so the correctness of $L_p(T)\bmod \ell$ follows from 
taking the characteristic polynomial of the resulting matrix.

\subsection{Complexity analysis}\label{Subsection-LV-complexity}

See Algorithm~\ref{alg:L-poly-mod-ell}. Note that $K_i\leq \ell^6$ and $\abs{I}\leq 6$. Also, 
the number of divisors of $n$ is $\exp(O(\log n/\log\log n))=n^{o(1)}$ by Theorem~1 in \cite{NR83}.
So the size of \textit{Degree-tries} is at most $(\ell^6)^{o(1)}\in\ell^{o(1)}$ as $\ell\to\infty$.
For the correct candidate given by the index $i$, we are certain that $K_i$ is a multiple of the torsion field
extension degree.

\begin{proposition}\label{prop:ell-torsion-complexity}
For a general prime $\ell\neq p$ and a trial degree $k\leq\ell^6$, let $n=k\log p$.
If \textit{Attempt-to-build-$\ell$-torsion} uses $8$ samples and stops Algorithm~4 of~\cite{Sut10}
as soon as its basis has $6$ elements, then it takes
\[
O\left(n\log n+\frac{n}{\log\ell}\ell^{5/2}+\ell^6\right)
\]
group operations. In particular, since $n=k\log p$ where $k\leq \ell^6$, it takes
\[
O\left(\ell^6\log p\,\log(\ell^6\log p)
+\frac{\ell^{17/2}\log p}{\log\ell}\right)
\]
group operations.
\end{proposition}
\begin{proof}
Once the basis has $6$ elements, their independent multiples of order $\ell$ form a basis for the full $\ell$-torsion.
Thus all discrete logarithms use a basis of rank at most $5$.
Let $H$ be the subgroup generated by the current basis, of order $\ell^a$, exponent $\ell^m$, and rank $r\leq5$.
The Hasse--Weil bound gives $\log|H|\leq\log|\Jac(C)(\F_{p^k})|=O(n)$, so $a=O(n/\log\ell)$ and $m\leq a$.
In Corollary~1 of~\cite{Sut10}, the logarithmic term is therefore $O(n\log n)$ and the other term is
$(a/r)\ell^{r/2}=O((n/\log\ell)\ell^{5/2})$. The same bound applies to Sutherland's Algorithm~3, 
the extended discrete-logarithm subroutine used within his Algorithm~4, as explained immediately after Algorithm~3 in~\cite{Sut10}.
With $8$ samples and at most $5$ rounds of discrete logarithms, Algorithm~4 makes only a constant number of these computations.
Cofactor multiplication, order checks, basis updates, and extracting the order-$\ell$ multiples take $O(n)$ group operations,
while constructing the $\ell^6$-element table takes $O(\ell^6)$ group operations. This proves the first bound, and
substituting $k\leq\ell^6$ gives the second bound.
\end{proof}

\begin{remark}
    In Theorem~4.2 of \cite{PR26}, Pizarro-Madariaga and Riquelme 
    prove that, for a hyperelliptic curve
    $C/\F_p$ of genus $3$ and a prime $\ell>2$ with $\ell\neq p$, the torsion 
    field extension degree $d$ satisfies 
    \[
    d\leq (\ell-1)(\ell^3+1)=\ell^4-\ell^3+\ell-1.
    \]
    This gives an $O(\ell^4)$ bound for the minimal torsion field extension degree in the hyperelliptic case,
    improving on the general $O(\ell^6)$ bound used above.
    Since our algorithm uses only $\ell\in\set{2,3}$, this refinement does not change the
    $(\log p)^{O(1)}$ expected running time of Step~4.
\end{remark}

By~\cite[Section~3]{Pak00}, when $\Jac(C)(\F_{p^k})$ contains the full
$\ell$-torsion, the probability of $8$ independent uniform elements generating its $\ell$-Sylow subgroup is at least $\frac12$.
Hence the expected number of passes is at most $2$, so multiplying the bound in Proposition~\ref{prop:ell-torsion-complexity}
by the $\ell^{o(1)}$ possible trial degrees bounds the expected number of group operations in Algorithm~\ref{alg:L-poly-mod-ell}.

By Proposition~\ref{prop:ell-2-3-suffice}, we only need $\ell=2,3$ to determine the unique candidate.
For these primes, the size of \textit{Degree-tries} is $O(1)$ with respect to $p$, and
Proposition~\ref{prop:ell-torsion-complexity} gives $(\log p)^{O(1)}$ group operations per attempt.
In this case, we are operating in the group $\Jac(C)(\F_{p^{k}})$, where each operation takes
$O(n\log^2 n)$ bit operations in the typical case and $O(n^2\log n)$ expected bit operations in the naive case,
with $n=k\log p=O(\log p)$ (see Section~\ref{section-analysis}). Computing the Frobenius map takes $O(\log p)$ field operations.

Therefore, Step 4 takes $(\log p)^{O(1)}=p^{o(1)}$ expected bit operations, and the overall lifting algorithm
still takes $O(p^{1/4+o(1)})$ expected time.

The space is dominated by the table \textit{hash-to-vec}, which stores all $\ell^6$ elements
of $\Jac(C)(\F_{p^k})[\ell]$ at $O(k\log p)$ bits each, for $O(\ell^6k\log p)$ bits in total, which is $O(\log p)$ for fixed $\ell$.

\section{Implementation and comparisons}\label{section-implementation}

Implementations of the hybrid group operation (Section~\ref{section-group-hybrid}) and 
the naive group operation (Section~\ref{section-group-naive}) in Magma \cite{MAGMA} (version V2.29-4) are available 
on \cite{Shi26}. The implementation of the $L$-polynomial lifting algorithm (Section~\ref{section-lifting})
is also on \cite{Shi26}, as is the implementation of Steps 3 and 4.
We use a C implementation of the algorithm described 
in \cite{CHS23} as a subprocess to compute the $L$-polynomial 
modulo~$p$. 

The following experiments compare our timings with a C++ implementation of Costa's algorithm \cite{Cos15, CHK19}\footnotemark
and a Magma implementation of Tuitman's algorithm \cite{Tui14, Tui16}.
\subsection{Experiment 1}
Table~\ref{Table-exp1} displays the average total time to compute the $L$-polynomials
for a single curve at all odd primes of good reduction up to $2^n$. 

To benchmark \cite{CHS23} and our work, we fixed $10$ randomly generated curves 
from Sutherland's database \cite{Sut19} and measured the CPU time to perform all $L$-polynomial 
computations for odd primes of good reduction up to $2^n$. For each $n$, we first computed 
$L_p(T)\bmod p$ for all odd primes of good reduction up to $2^n$ using the average polynomial-time algorithm in \cite{CHS23}. For primes $p<128$, we used 
a point counting algorithm to compute the full $L$-polynomial. For primes $128<p<500$, we used 
a combination of point counting and our lifting algorithm. For $p>500$, we only used our lifting algorithm. We then removed the $2$ fastest and 
the $2$ slowest entries, and averaged the remaining $6$ entries. We parallelized this computation 
using $10$ cores. 

To benchmark Tuitman's algorithm \cite{Tui14, Tui16} and Costa's algorithm \cite{Cos15, CHK19}, for each $n\geq 7$, we generated $10$
random $(C, p)$ pairs ($5$ pairs for $n\geq 21$) where $C$ is a smooth plane quartic over $\Q$ (from \cite{Sut19}) and $p\in [2^{n-1}, 2^{n})$ is an odd prime
of good reduction. We then removed the $20\%$ fastest and $20\%$ slowest entries, took the average,
multiplied it by the estimated number of odd primes of good reduction in this interval, and summed the results accordingly. We note that for $n=22$, 
only one instance of Tuitman's algorithm completed successfully.

\footnotetext{Note that Costa's algorithm is also capable of computing only $L_p(T)\bmod p$ instead of the full $L$-polynomial, though  
experiments show that \cite{CHS23} is faster.}

\begin{table}[!ht]
    \centering\footnotesize
    $\begin{array}{|r|r|r|r|r|r|}
        \hline
        n & \text{Compute all }& \text{Lift all }L_p(T) & \text{Compute all }  & \text{Compute all }  &  \text{Compute all}\\
          &  L_p(T) \bmod p& \bmod \ p  &  L_p(T) \text{ (\cite{CHS23}}  & L_p(T)\text{ \cite{Cos15, CHK19}\footnotemark} &L_p(T)\text{ \cite{Tui14, Tui16}}  \\ 
          &\text{ (\cite{CHS23})} &\text{ (our work)} &+ \text{ our work)} & \text{(estimated)} & \text{(estimated)}\\
          \hline
            11 & 0.408 & 54.7\phantom{00} & 55.2\phantom{00} & 34.7& 1930 \\
            12 & 0.655 & 95\phantom{.000} & 95.7\phantom{00} & 108\phantom{.0} & 6600 \\
            13 & 1.28\phantom{0} & 157\phantom{.000} & 158\phantom{.000} &  360\phantom{.0}& 28200 \\
            14 & 2.96\phantom{0} & 291\phantom{.000} & 294\phantom{.000} & 1140\phantom{.0} & 113000 \\
            15 & 7.44\phantom{0} & 531\phantom{.000} & 539\phantom{.000} & 4220\phantom{.0} & 451000 \\
            16 & 16.4\phantom{00} & 1050\phantom{.000} & 1070\phantom{.000} & 15100\phantom{.0} & 2490000 \\
            17 & 35.2\phantom{00} & 2150\phantom{.000} & 2180\phantom{.000} & 53800\phantom{.0} & 12700000 \\
            18 & 84.3\phantom{00} & 4510\phantom{.000} & 4590\phantom{.000} & 188000\phantom{.0} & 54600000 \\
            19 & 171\phantom{.000} & 9700\phantom{.000} & 9870\phantom{.000} & 812000\phantom{.0} & 269000000 \\
            20 & 427\phantom{.000} & 20900\phantom{.000} & 21300\phantom{.000} & 3010000\phantom{.0} & 1660000000 \\
            21 & 1030\phantom{.000} & 45600\phantom{.000} & 46600\phantom{.000} & 12300000\phantom{.0} & 11000000000 \\
            22 & 2420\phantom{.000} & 93500\phantom{.000} & 95900\phantom{.000} & 245000000\phantom{.0} & 49200000000 \\
          \hline
        \end{array}$ 
        \bigskip
\caption{Average computation time (in 2.000 GHz AMD EPYC 7713 core-seconds) for computing the $L$-polynomials of a fixed curve
for all odd primes of good reduction up to $2^{n}$.
}\label{Table-exp1}
\end{table}

\footnotetext{The jump from $n=21$ to $n=22$ is due to the fact that 
once $p^3>2^{64}$, i.e., $p>2^{64/3}\approx 2^{21.3}$, elements of $\Z/p^3\Z$ need two machine words to be stored. See pages~28--29 of \cite{Cos15}.}

\subsection{Experiment 2}
Table~\ref{Table-exp2} displays the average CPU time 
to obtain the $L$-polynomial for a single curve at a single odd
 prime of good reduction in the interval $[2^{n-1}, 2^{n})$. 

To obtain these results, for each $n$, we selected a random curve from Sutherland's database \cite{Sut19}
and a random odd prime of good reduction in the interval $[2^{n-1}, 2^{n})$, and computed the $L$-polynomial. 
We used the $O(p^{1/2+o(1)})$ time algorithm in \cite{CHS23} to compute $L_p(T)\bmod p$ and used 
our work to lift $L_p(T)\bmod p$.
For the combined implementation of \cite{CHS23} and our work, we repeated this experiment $25$ times. For Costa's algorithm and Tuitman's algorithm, 
we repeated this experiment $10$ times for $n\leq 20$ and $5$ times for $n\geq 21$. For $n=22$, only 
one instance of Tuitman's algorithm completed successfully.
We then removed the $20\%$ fastest and $20\%$ slowest timings and computed the average of the remaining data.

\begin{table}[!ht] 
    \centering\footnotesize
    $\begin{array}{|r|r|r|r|r|r|}
        \hline
      n   & \text{Compute }       & \text{Lift }            &  \text{Compute } L_p(T) & \text{Compute } L_p(T)    & \text{Compute } L_p(T) \\
          & L_p(T)  \bmod p       &  L_p(T)  \bmod p   &\text{ (\cite{CHS23}}      &  \text{\cite{Cos15, CHK19}\footnotemark}& \text{\cite{Tui14, Tui16}} \\ 
          &\text{ (\cite{CHS23})}& \text{ (our work)}      &+ \text{ our work)}       &                               & \\
          \hline
12 & 0.029 & 0.172 & 0.201 & 0.288  & 18.4\phantom{00} \\
13 & 0.031 & 0.159 & 0.190 & 0.545 & 46.7\phantom{00} \\
14 & 0.035 & 0.183 & 0.219 & 0.897 & 97.5\phantom{00} \\
15 & 0.036 & 0.195 & 0.231 & 1.91\phantom{0} & 210\phantom{.000} \\
16 & 0.044 & 0.219 & 0.263 & 3.58\phantom{0} & 672\phantom{.000} \\
17 & 0.045 & 0.253 & 0.299 & 6.79\phantom{0} & 1780\phantom{.000} \\
18 & 0.054 & 0.265 & 0.319 & 12.5\phantom{00} & 3910\phantom{.000} \\
19 & 0.080 & 0.294 & 0.374 & 30.6\phantom{00} & 10500\phantom{.000} \\
20 & 0.101 & 0.328 & 0.429 & 56.9\phantom{00} & 36100\phantom{.000} \\
21 & 0.148 & 0.395 & 0.543 & 126\phantom{.000} & 127000\phantom{.000} \\
22 & 0.215 & 0.443 & 0.658 & 1660\phantom{.000} & 272000\phantom{.000} \\
23 & 0.350 & 0.507 & 0.857 & 8750\phantom{.000} & - \\
24 & 0.503 & 0.609 & 1.11\phantom{0} & 20900\phantom{.000} & - \\
25 & 0.897 & 0.709 & 1.61\phantom{0} & 41700\phantom{.000} & - \\
26 & 1.42\phantom{0} & 0.866 & 2.28\phantom{0}  & - & - \\
27 & 2.35\phantom{0} & 1.07\phantom{0} & 3.42\phantom{0} & - & - \\
28 & 3.57\phantom{0} & 1.28\phantom{0} & 4.85\phantom{0} & - & - \\
29 & 6.37\phantom{0} & 1.5\phantom{00} & 7.87\phantom{0} & - & - \\
30 & 9.79\phantom{0} & 1.78\phantom{0} & 11.6\phantom{00} & - & - \\
31 & 15.7\phantom{00} & 3.12\phantom{0} & 18.8\phantom{00} & - & - \\
32 & 20.7\phantom{00} & 3.9\phantom{00} & 24.6\phantom{00} & - & - \\

          \hline
    \end{array}$
    \bigskip
\caption{Average computation time (in 2.000 GHz AMD EPYC 7713 core-seconds) for computing the $L$-polynomial 
of a random genus $3$ smooth plane quartic at a random odd prime of good reduction, $p\in [2^{n-1}, 2^{n})$.
}\label{Table-exp2}
\end{table}
\footnotetext{See the previous footnote.}

\newpage
    \printbibliography[
        heading=bibintoc
        ]

\end{document}